\documentclass[12pt,a4paper]{article}

\usepackage{amsbsy,amsfonts,amsmath,amssymb,amsthm}
\usepackage{array}
\usepackage{bm}
\usepackage{color}
\usepackage{enumerate}
\usepackage{mathrsfs}
\usepackage{latexsym}
\usepackage[colorlinks=true]{hyperref}
\usepackage{mathtools}
\mathtoolsset{showonlyrefs, showmanualtags}
\hypersetup{urlcolor=black, citecolor=black, linkcolor=black}
\definecolor{wineRed}{rgb}{0.7,0,0.3}
\definecolor{grandBleu}{rgb}{0,0,0.8}
\definecolor{darkGreen}{rgb}{0,0.4,0}
\definecolor{blueViolet}{rgb}{0.4,0,1.0}
\definecolor{bloodOrange}{rgb}{0.85,0.05,0}
\definecolor{mycolor}{rgb}{0.8,0,0.2}
\usepackage[textsize=small]{todonotes}
\usepackage{cite}

\usepackage{comment}

\DeclareMathAlphabet{\mathpzc}{OT1}{pzc}{m}{it}

\usepackage[textsize=small]{todonotes}
\numberwithin{equation}{section}

\theoremstyle{plain}
\newtheorem{mTh}{Main Theorem} 
\newtheorem{lem}{Lemma}[section]

\theoremstyle{definition}
\newtheorem{defn}{Definition}%[section]
\newtheorem{thm}{Theorem}%[section]
\newtheorem{rem}{Remark}

\newtheorem{ex}{Example}

\def\B{\mathbb{B}}
\def\N{\mathbb{N}}
\def\R{\mathbb{R}}
\def\D{\mathscr{D}}
\def\L{\mathcal{L}}
\def\H{\mathcal{H}}
\def\sH{{L^2(0,T;H)}}

\def\ds{\displaystyle}

\def\Lap{\mathit{\Delta}}
\def\Sgn{\mathop{\mathrm{Sgn}}\nolimits}
\def\supp{\mathop{\mathrm{supp}}}
\DeclareMathOperator{\essinf}{ess\inf}

\DeclareMathOperator{\diver}{div}
\newcommand{\trace}[1]{\left.{#1}\right|_\Gamma}

\begin{document}
\vspace*{0.5cm}
\begin{center}
    \textbf{\large A CLASS OF INITIAL-BOUNDARY VALUE PROBLEMS 
    \\[0.5ex]
    GOVERNED BY  PSEUDO-PARABOLIC 
    \\[1ex]
    WEIGHTED TOTAL VARIATION FLOWS}\footnotemark[1]
\end{center}
\vspace{1ex}
\begin{center}
    \textit{Dedicated to Professor Nobuyuki Kenmochi on the occasion of his 77th birthday}
\end{center}
\vspace{3ex}
\begin{center}
    \textsc{Toyohiko Aiki}
    \\
%\footnotemark[2]
Department of Mathematics, Faculty of Sciences, \\
Japan Women's University, 
\\
    2--8--1, Mejirodai,Bunkyo-ku, Tokyo 112--8681, Japan
    \\
{\ttfamily aikit@fc.jwu.ac.jp}
\vspace{1ex}
\end{center}
\begin{center}
    \textsc{Daiki Mizuno}
\\
    Division of Mathematics and Informatics, 
    \\
    Department of Mathematics and Informatics, \\ Graduate School of Science and Engineering, Chiba Univercity, \\ 1--33, Yayoi-cho, Inage-ku, 263--8522, Chiba, Japan
\\
    {\ttfamily d-mizuno@chiba-u.jp}
\vspace{1ex}
    \end{center}
    \begin{center}
        \textsc{Ken Shirakawa}
        \\
Department of Mathematics, Faculty of Education, Chiba University \\ 1--33 Yayoi-cho, Inage-ku, 263--8522, Chiba, Japan
\\
{\ttfamily sirakawa@faculty.chiba-u.jp}
    \end{center}
\footnotetext[1]{
This work is supported by  Grant-in-Aid for Scientific Research (C) No. 20K03672, JSPS.
%74N05, % Crystals in solids
%74N20. % Dynamics of phase boundaries in solids
35J62.%  	Quasilinear elliptic equations
}
%\footnotetext[1]{aaa}
%\footnotetext[2]{bbb}
\vspace{5ex}

\noindent
{\bf Abstract.}
In this paper, we consider a class of initial-boundary value problems governed by pseudo-parabolic total variation flows. The principal characteristic of our problem lies in the velocity term of the diffusion flux, a feature that can bring about stronger regularity than what is found in standard parabolic PDEs. Meanwhile, our total variation flow contains singular diffusion, and this singularity may lead to a degeneration of the regularity of solution. The objective of this paper is to clarify the power balance between these conflicting effects. Consequently, we will present mathematical results concerning the well-posedness and regularity of the solution in the Main Theorems of this paper.
%\pagebreak
\newpage

\section*{Introduction}
This paper is devoted to the study of a class of initial-boundary value problems of pseudo-parabolic PDEs. Each initial-boundary value problem is denoted by $ \mathrm{(P)}_\varepsilon $, with a constant $ \varepsilon \geq 0 $, and formulated as:
\begin{align*}
    \mathrm{(P)}_\varepsilon \hspace{4ex} 
    %& 
    %\\
    & \begin{cases}
        \ds \partial_t u -\mathrm{div} \bigl( \alpha(x) \partial \gamma_\varepsilon(\nabla u) +\beta(x) \nabla \partial_t u \bigr) \ni 0 ~ \mbox{ in $ Q $,}
        \\[1ex]
        \bigl( \alpha(x) \partial \gamma_\varepsilon(\nabla u) +\beta(x) \nabla \partial_t u \bigr) \cdot n_\Gamma \ni 0 ~ \mbox{ on $ \Sigma $,}
        \\[1ex]
        u(0, x) = u_0(x), ~~ x \in \Omega.
    \end{cases}
\end{align*}
The class of initial-boundary value problems $\bigl\{\mbox{(P)$_\varepsilon$}\bigr\}_{\varepsilon \geq 0}$ is considered under the following notations and assumptions: 
\begin{description}
    \item[(A0)]$0 < T < \infty$ and $N \in \N$ are fixed constants of time and spatial dimension, respectively. 
    \item[(A1)]$ \Omega \subset \R^N $ is a bounded domain such that the boundary $ \Gamma := \partial \Omega $ is smooth ($ C^\infty $-class) when $ N > 1 $. Also, $ n_\Gamma $ is the unit outer normal on $ \Gamma $. Additionally, we let:
        \begin{align*}
            & Q := (0, T) \times \Omega, ~~ \Sigma := (0, T) \times \Gamma,
            \\
            & ~~ H :=  L^2(\Omega), ~~ \mbox{and} ~~
            V := H^1(\Omega).
        \end{align*}
    \item[(A2)] $ \alpha \in H^1(\Omega) \cap L^\infty(\Omega) $, $ \beta \in W^{1, \infty}(\Omega) $ and $f \in \sH$ are fixed functions, such that:
        \begin{center}
            $ \essinf \alpha(\Omega) \geq 0 $, and $ \delta_* := \inf \beta(\Omega) > 0 $. 
        \end{center}
    \item[(A3)]$ \{ \gamma_\varepsilon \}_{\varepsilon \geq 0} $ is a class of convex function on $ \R^N $, defined as:
        \begin{align*}
            & \gamma_\varepsilon: \R^d \ni y \mapsto \gamma_\varepsilon(y) := \sqrt{\varepsilon^2 + |y|^2} \in [0,\infty), ~\mbox{for $ \varepsilon \geq 0 $.} 
        \end{align*}
        Also, for any $ \varepsilon \geq 0 $, $ \partial \gamma_\varepsilon \subset \R \times \R $ denotes the subdifferential of $ \gamma_\varepsilon $.
    \item[(A4)]$ u_0 \in H $ is a fixed function such that:
        \begin{align*}
            & u_0 \in W_0 := \left\{ \begin{array}{l|l}
                w \in H^2(\Omega) & \nabla w \cdot n_\Gamma = 0 ~\mbox{in $ H^{\frac{1}{2}}(\Gamma) $}
            \end{array} \right\}.
        \end{align*}
\end{description}
Besides, for each $ \varepsilon \geq 0 $, the solution to the problem (P)$_\varepsilon$ is defined in the following weak (variational) sense.
\begin{defn}\label{Def.sol}
    For any $ \varepsilon  \in [0, 1] $, a function $ u : [0, T] \longrightarrow H $ is called a solution to (P)$_\varepsilon$, iff.
    \begin{align}\label{defOfSol00}
        & u \in W^{1, 2}(0, T; V), \mbox{ with $ u(0) = u_0 $ in $ H $,}
    \end{align}
    and
    \begin{align}\label{defOfSol01}
        \int_\Omega \partial_t u(t) & (u(t) -\varphi) \, dx + \int_\Omega \beta(x) \nabla \partial_t u(t) \cdot \nabla \bigl( u(t) -\varphi \bigr) \, dx +\int_\Omega \alpha(x) \gamma_\varepsilon(\nabla u(t)) \, dx 
        \\[1ex]
        & \leq \int_\Omega \alpha(x) \gamma_\varepsilon(\nabla \varphi) \, dx, ~ \mbox{ for any $ \varphi \in V $ and for a.e. $t \in (0,T)$}.
    \end{align}
\end{defn}

The class of problems $ \{ \mathrm{(P)}_\varepsilon \}_{\varepsilon \geq 0} $ is motivated to establish a theory of pseudo-parabolic version for the weighted total variation flow:
\begin{align}\label{TVF}
    %\mathrm{(TVF)} \hspace{4ex} 
    & \begin{cases}
        \ds \partial_t u -\mathrm{div} \left( \alpha (x) \frac{Du}{|Du|} \right) = 0 ~ \mbox{ in $ Q $,}
        \\[1ex]
        \bigl( \alpha(x) \frac{Du}{|Du|} \bigr) \cdot n_\Gamma = 0 ~ \mbox{ in $ \Sigma $,}
        \\[1ex]
        u(0, x) = u_0(x), ~ \mbox{$ x \in \Omega $,}
    \end{cases}
\end{align}
which has been studied by a lot of mathematicians, as one of key-problems of mathematical models of image denoising processes \cite{MR2028858,MR2170510,MR3395125},  grain boundary motions \cite{MR2096945,MR2746654,MR1865089,MR1712447,giga2023fractional}, phase-transitions \cite{MR1847840,MR1851862}, and so on. According to the previous works, it is known that the weak solution to \eqref{TVF} admits the following regularity:
\begin{align}\label{regTVF}
    u & \in W^{1, 2}(0, T; H) ~\mbox{ and }~ |u(\cdot)|_{BV(\Omega)} \in L^\infty(0, T), 
    \\
    & \mbox{whenever $ \essinf \alpha(\Omega) > 0 $ and $ u_0 \in BV(\Omega) \cap H $. }
    \nonumber
\end{align}
Comparing \eqref{defOfSol00} with \eqref{regTVF}, we can observe a certain regularization effect of the velocity of flux $ \beta(x) \nabla \partial_t u $ as in (P)$_\varepsilon$.
However, in this paper, our interest will be in more fine regularity than \eqref{defOfSol00}. 

Actually, nowadays, we can find a number of mathematical researches \cite{MR0437936,MR0330774,MR0466818,MR2263926}, which deal with pseudo-parabolic problems based on linear/nonlinear PDEs, and from some of these, we can also see the regularity property comparable with:
\begin{align}\label{reg_pLap}
    & u \in W^{1, 2}(0, T; W^{2, p}(\Omega)) ~\mbox{ for some $ 1 < p < \infty $.}
\end{align}

In our problem (P)$_\varepsilon$, if $ \varepsilon > 0 $, then the diffusion flux $ \partial \gamma_\varepsilon(\nabla u) $ as in (P)$_\varepsilon$ is described in a smooth quasilinear form  $  \alpha \frac{\nabla u}{\sqrt{\varepsilon^2 +|\nabla u|^2}} $, and the smoothness of flux would lead to some strong regularity similar to \eqref{reg_pLap}. 
But while, if $ \varepsilon = 0 $, then the corresponding diffusion flux $ \alpha \frac{Du}{|Du|}  $ contains a singularity, and this singularity would bring down some degeneration for the regularization effect of pseudo-parabolicity. 

The objective of this paper is to clarify the power balance between these conflicting effects. In view of this, we have set the goal to prove the following two Main Theorems. 
\begin{description}
    \item[Main Theorem 1:]the result for the problem (P)$_\varepsilon$ when $ \varepsilon > 0 $, i.e. the smooth case, which is based on the regularity $ u \in W^{1, 2}(0, T; W_0) $ of the solution $u$. 
    \item[Main Theorem 2:]the result for the problem (P)$_0$, i.e. the singular case when $ \varepsilon = 0 $, which is to verify the regularity $ u \in L^\infty(0, T; W_0) $ of the solution $ u $, and to clarify the rigorous mathematical meaning of the set-valued component $ \alpha(x) \partial \gamma_0(\nabla u) $.
\end{description}

The content of this paper is as follows. Preliminaries are given in Section 1, and on this basis, the Main Theorems are stated in Section 2. For the proofs of Main Theorems, we prepare Section 3 to setting up of the regularity theory for an auxiliary elliptic boundary value problem. The auxiliary problem is associated with the time-discretization scheme of our problem (P)$_\varepsilon$. Hence, essentially, Section 3 will have a key role to underpin the theoretical part of this work. Based on these, the Main Theorems are proved in Section 4, by means of the auxiliary results obtained in Section 3, and appendix in Section 5.

\section{Preliminaries}
We begin by prescribing the notations used throughout this paper. 
\bigskip

\noindent
\underline{\textbf{\textit{Notations in real analysis.}}}
We define:

\begin{align*}
    & r \vee s := \max \{ r, s \} ~ \mbox{ and } ~ r \wedge s := \min \{r, s\}, \mbox{ for all $ r, s \in [-\infty, \infty] $,}
\end{align*}
and especially, we write:
\begin{align*}
    & [r]^+ := r \vee 0 ~ \mbox{ and } ~ [r]^- := -(r \wedge 0), \mbox{ for all $ r \in [-\infty, \infty] $.}
\end{align*}

Let $ d \in \N $ be a fixed dimension. We denote by $ |y| $ and $ y \cdot z $ the Euclidean norm of $ y \in \mathbb{R}^d $ and the scalar product of $ y, z \in \R^d $, respectively, i.e., 
\begin{equation*}
\begin{array}{c}
| y | := \sqrt{y_1^2 +\cdots +y_d^2} \mbox{ \ and \ } y \cdot z  := y_1 z_1 +\cdots +y_d z_d, 
\\[1ex]
\mbox{ for all $ y = [y_1, \ldots, y_d], ~ z = [z_1, \ldots, z_d] \in \mathbb{R}^d $.}
\end{array}
\end{equation*}
Besides, we let:
\begin{align*}
    & \mathbb{B}^d := \left\{ \begin{array}{l|l}
        y \in \R^d & |y| < 1
    \end{array} \right\} ~ \mbox{ and } ~ \mathbb{S}^{d -1} := \left\{ \begin{array}{l|l}
        y \in \R^d & |y| = 1
    \end{array} \right\}.
\end{align*}
We denote by $\mathcal{L}^{d}$ the $ d $-dimensional Lebesgue measure, and we denote by $ \mathcal{H}^{d} $ the $ d $-dimensional Hausdorff measure.  In particular, the measure theoretical phrases, such as ``a.e.'', ``$dt$'', and ``$dx$'', and so on, are all with respect to the Lebesgue measure in each corresponding dimension. Also on a Lipschitz-surface $ S $, the phrase ``a.e.'' is with respect to the Hausdorff measure in each corresponding Hausdorff dimension. In particular, if $S$ is $C^1$-surface, then we simply denote by $dS$ the area-element of the integration on $S$.

For a Borel set $ E \subset \R^d $, we denote by $ \chi_E : \R^d \longrightarrow \{0, 1\} $ the characteristic function of $ E $. Additionally, for a distribution $ \zeta $ on an open set in $ \R^d $ and any $i \in \{ 1,\dots,d \}$, let $ \partial_i \zeta$ be the distributional differential with respect to $i$-th variable of $\zeta$. As well as we consider, the differential operators, such as $\nabla,\ \diver, \ \nabla^2$, and so on, are considered in distributional senses.
\bigskip

\noindent
\underline{\textbf{\textit{Abstract notations. (cf. \cite[Chapter II]{MR0348562})}}}
For an abstract Banach space $ X $, we denote by $ |\cdot|_{X} $ the norm of $ X $, and denote by $ \langle \cdot, \cdot \rangle_X $ the duality pairing between $ X $ and its dual $ X^* $. In particular, when $ X $ is a Hilbert space, we denote by $ (\cdot,\cdot)_{X} $ the inner product of $ X $. 

For two Banach spaces $ X $ and $ Y $,  let $  \mathscr{L}(X; Y)$ be the Banach space of bounded linear operators from $ X $ into $ Y $. 

For Banach spaces $ X_1, \dots, X_d $ with $ 1 < d \in \N $, let $ X_1 \times \dots \times X_d $ be the product Banach space endowed with the norm $ |\cdot|_{X_1 \times \cdots \times X_d} := |\cdot|_{X_1} + \cdots +|\cdot|_{X_d} $. However, when all $ X_1, \dots, X_d $ are Hilbert spaces, $ X_1 \times \dots \times X_d $ denotes the product Hilbert space endowed with the inner product $ (\cdot, \cdot)_{X_1 \times \cdots \times X_d} := (\cdot, \cdot)_{X_1} + \cdots +(\cdot, \cdot)_{X_d} $ and the norm $ |\cdot|_{X_1 \times \cdots \times X_d} := \bigl( |\cdot|_{X_1}^2 + \cdots +|\cdot|_{X_d}^2 \bigr)^{\frac{1}{2}} $. In particular, when all $ X_1, \dots,  X_d $ coincide with a Banach space $ Y $, the product space $X_1 \times \dots \times X_d$ is simply denoted by $[Y]^d$.
%\bigskip

\begin{rem}\label{rem:NS}
    Due to the smoothness and compactness of $\Gamma$, we may suppose the existence of a finite number $ M \in \N $ and a finite open covering $ \{ U_\ell \}_{\ell = 0}^M $ of $ \Omega $, which fulfill the following conditions. 
%    class of smooth function $\{ \eta_l \}_{l=1}^M \subset C^\infty(\overline{\Omega}), \ \{ a_l \}_{l=1}^M \subset C^\infty(\overline{\Omega}), \ \{ r_l \}_{l = 1}^M \subset \R, \ \{ h_l \}_{l = 1}^M \subset \R$ satisfying the following conditions:
  \begin{itemize}
      \item 
      $ \displaystyle \Gamma \subset \bigcup_{\ell=1}^M U_\ell $ and $ \ds \Omega \setminus \bigcup_{\ell=1}^M U_\ell \subset U_0 \subset \overline{U_0} \subset \Omega $, and the distance function 
      \begin{equation}
        d_\Gamma :\, x \in \R^N \mapsto d_\Gamma(x) := \inf_{y \in \Gamma} |y-x| \in [0,\infty)
      \end{equation}
      is $C^\infty$-function on the covering $\ds{\bigcup_{\ell=1}^M U_\ell }$ of $\Gamma$.
  \item There exists a finite set $ \{ r_\ell \}_{\ell = 1}^M \subset (0,\infty)$,  a finite class of functions $ \{ a_\ell \}_{\ell = 1}^M \subset C^\infty(r_\ell \mathbb{B}^{N -1}) $, and a class of congruent transforms $ \{ \Theta_\ell \}_{\ell = 1}^M $ such that:
      \begin{align*}
          \Gamma \cap U_\ell ~& =  \Theta_\ell \left\{ \begin{array}{l|l}
              [z', a_\ell(z')] \in \mathbb{R}^N & 
               \parbox{5.25cm}{$ z' = [z_1, \dots, z_{N -1}] \in r_\ell \mathbb{B}^{N -1} $}
          \end{array} \right\}.
      \end{align*}
      \item for any $ \ell \in \{1, \dots, M\} $, there exists a (small) positive constant $ h_\ell $, such that $ U_\ell $, $ U_\ell \cap \Omega $, and $ U_\ell \cap \overline{\Omega}^\mathrm{\,C} $ are expressed as:
          \begin{align*}
              & \begin{cases}
                  U_\ell = \Theta_\ell  \Xi_\ell W_\ell, \mbox{ with } W_\ell := r_\ell \mathbb{B}^{N -1} \times (-h_\ell, h_\ell),
                  \\[1ex]
                  U_\ell \cap \Omega = \Theta_\ell \Xi_\ell W_\ell^+, \mbox{ with } W_\ell^+ := r_\ell \mathbb{B}^{N -1} \times (0, h_\ell),
                  \\[1ex]
                  U_\ell \cap \overline{\Omega}^\mathrm{\,C} = \Theta_\ell \Xi_\ell W_\ell^-, \mbox{ with } W_\ell^- :=  r_\ell \mathbb{B}^{N -1} \times (-h_\ell, 0),
              \end{cases}
          \end{align*}
          by using the following $ C^\infty $-diffeomorphism: 
          \begin{gather}
              \Xi_\ell: z = [z', z_N] \in W_\ell \mapsto \bigl[ z',  a_\ell (z') \bigr] -z_N \Theta_\ell^{-1} n_\Gamma \in \mathbb{R}^N,
              \\
              \mbox{with $ z' = [z_1, \dots, z_{N -1}] \in \mathbb{B}^{N -1} $.}
          \end{gather}
%    \item $\{ \eta_l \}_{l=0}^M$ is the partition of unity for $\{ U_l \}_{l=0}^M$, which satisfies:
  \end{itemize}
  Note that when $z_N = 0$, the inverse matrix of Jacobian $(D \Xi_\ell)^{-1}$ can be easily calculated as follows:
          \begin{equation}
            \hspace{-0.15cm}(D \Xi_\ell)^{-1}([z',0]) = \left( \hspace{-0.25cm}\begin{array}{ccccc}
              1-\tilde n_{\Gamma,1}^2 & -\tilde n_{\Gamma,2} \tilde n_{\Gamma,1} &\cdots & -\tilde n_{\Gamma,N-1} \tilde n_{\Gamma,1} & -\tilde n_{\Gamma,N} \tilde n_{\Gamma,1} \\
              -\tilde n_{\Gamma,1} \tilde n_{\Gamma,2} & 1 -\tilde n_{\Gamma,2}^2 &\cdots & -\tilde n_{\Gamma,N-1} \tilde n_{\Gamma,2} & -\tilde n_{\Gamma,N} \tilde n_{\Gamma,2} \\
              \vdots & & \ddots & & \vdots \\
              -\tilde n_{\Gamma,1} \tilde n_{\Gamma,N-1} & -\tilde n_{\Gamma,2}\tilde n_{\Gamma,N-1} &\cdots & 1-\tilde n_{\Gamma,N-1}^2 & -\tilde n_{\Gamma,N} \tilde n_{\Gamma,N-1} \\
              \tilde n_{\Gamma,1} & \tilde n_{\Gamma,2} & \cdots & \tilde n_{\Gamma,N-1} & \tilde n_{\Gamma,N}
            \end{array} \hspace{-0.25cm}\right), \label{D_Xi}
          \end{equation}
          where $\tilde n_\Gamma := \Theta_\ell^{-1} n_\Gamma$ and $\tilde n_{\Gamma,j}$ is the $j$-th element of $\tilde n_\Gamma$.
  As is well-known, $n_\Gamma = \nabla d_\Gamma|_\Gamma$ on $\Gamma$, and hence, we can say $\partial_i n_\Gamma = \partial_i \nabla d_\Gamma|_\Gamma = \nabla \partial_i d_\Gamma|_\Gamma$, for $i = 1,\dots,N$. Moreover, if $u \in H^2(U_\ell \cap \Omega)$, then it holds that:
    \begin{gather}\label{bc0}
      \nabla u \cdot n_\Gamma(x) = - \bigl[\partial_{z_N} (u \circ\Theta_\ell \circ \Xi_\ell)\bigr](z',0), \mbox{ for a.e. } x \in U_\ell \cap \Gamma \mbox{ and } z' \in r_\ell \B^{N-1},
      \\
      \mbox{ satisfying } x = (\Theta_\ell \circ \Xi_\ell) [z',0], \mbox{ and for all } \ell = 1,\dots,M.
    \end{gather}
\end{rem}

\noindent
\underline{\textbf{\textit{Notations in convex analysis.}}}
Let $X$ be an abstract Hilbert space $X$. For a proper, lower semi-continuous (l.s.c.), and convex function $\Psi : \,X \longrightarrow (-\infty, \infty]$ on a Hilbert space $X$, we denote by $D(\Psi)$ the effective domain of $\Psi$. Also, we denote by $\partial \Psi$ the subdifferential of $\Psi$. The set $D(\partial \Psi) := \left\{ z \in X\,|\, \partial \Psi(z) \neq \emptyset \right\}$ is called the domain of $\partial\Psi$. The subdifferential $ \partial \Psi $ is known as a maximal monotone graph in the product space $X \times X$. We often use the notation ``$[z_0, z_0^*] \in \partial \Psi ~{\rm in}~ X \times X$", to mean that ``$z_0^* \in \partial \Psi(z_0) ~{\rm in}~ X~{\rm for}~ z_0 \in D(\partial \Psi)$", by identifying the operator $\partial \Psi$ with its graph in $X\times X$.
\medskip
\begin{ex}%[Examples of the subdifferential]
    \label{exConvex}
    For the sequence of real convex functions $\{ \gamma_\varepsilon \}_{\varepsilon \geq 0}$ as in (A3), the following items hold. 
    \begin{description}
        \item[(O)]The subdifferential $ \partial \gamma_0 \subset \R^N \times \R^N $ of the convex function $ \gamma_0 : y = [y_1, \dots, y_n] \in \R^N \mapsto |y| =  \sqrt{y_1^2 + \dots +y_N^2} \in [0, \infty) $ coincides with the following set-valued function $ \Sgn: \R^N \rightarrow 2^{\mathbb{R}^N} $, which is defined as:
\begin{align}\label{Sgn^d}
\Sgn :  y = [y_1, & \dots, y_N] \in \mathbb{R}^N \mapsto \Sgn(y) = \Sgn(y_1, \dots, y_N) 
  \nonumber
  \\
  & := \left\{ \begin{array}{ll}
          \multicolumn{2}{l}{
                  \ds \frac{y}{|y|} = \frac{[y_1, \dots, y_N]}{\sqrt{y_1^2 +\cdots +y_N^2}}, ~ } \mbox{if $ y \ne 0 $,}
                  \\[3ex]
          \overline{\mathbb{B}^N}, & \mbox{otherwise.}
      \end{array} \right.
  \end{align}
    %where $ \mathbb{D}^d $ denotes the closed unit ball in $ \mathbb{R}^d $ centered at the origin. 
    %Indeed, the set-valued function $ \Sgn^d $ coincides with the subdifferential of the convex function $ \gamma_0 : y \in \mathbb{R}^d \mapsto \gamma_0(y) = \sqrt{y_1^2 + \cdots +y_d^2} \in [0, \infty) $, i.e.:
%\begin{equation*}
%\partial |{}\cdot{}|(\xi) = \Sgn^d(\xi), \mbox{ for any $ \xi \in D(\partial |{}\cdot{}|) = \mathbb{R}^d $,}
%\end{equation*}
%and furthermore, it is observed that:
%\begin{equation*}
%  \partial  |{}\cdot{}|(0) = \mathbb{D}^d \begin{array}{c} \subseteq_{\hspace{-1.25ex}\mbox{\tiny$_/$}}  
%  \end{array} [-1, 1]^d 
%      = \bigl[ \partial_{\xi_1}  |{}\cdot{}| \times \cdots \times \partial_{\xi_d}  |{}\cdot{}| \bigr](0).
%\end{equation*}
\item[(\,I\,)]For every $ \varepsilon > 0 $, the subdifferential $\partial \gamma_\varepsilon$ is identified with the (single-valued) usual gradient, i.e.:
\begin{equation}
    D(\partial \gamma_\varepsilon) = \R^N ~and~\nabla \gamma_\varepsilon : \R^N \ni y \mapsto \nabla \gamma_\varepsilon(y) := \frac{y}{\sqrt{\varepsilon^2 + |y|^2}} \in \R^N.
\end{equation}
Moreover, since:
    \begin{align*}
        \gamma_\varepsilon(y) = \bigl| [\varepsilon, y] \bigr|_{\R^{N +1}} & =\bigl| [\varepsilon, y_1, \dots, y_N] \bigr|_{\R^{N +1}}, \mbox{ for all $ [\varepsilon, y] = [\varepsilon, y_1, \dots, y_N] \in \R^{N +1} $,}
        \\
        & \mbox{with $ \varepsilon \geq 0 $ and $ y = [y_1, \dots, y_N] \in \R^N $,}
    \end{align*}
    it will be estimated that:
%    \begin{subequations}\label{exM}
%        \begin{align}\label{exM00}
%            |\gamma_\varepsilon(y) -\gamma_{\tilde{\varepsilon}}(\tilde{y})| \leq & \bigl| [\varepsilon, y] -[\tilde{\varepsilon}, \tilde{y}] \bigr|_{\R^{N +1}} \leq |\varepsilon -\tilde{\varepsilon}| +|y -\tilde{y}|_{\R^N}, 
%            \nonumber
%            \\
%            & \mbox{for all $ \varepsilon, \tilde{\varepsilon} \geq 0  $ and $ y, \tilde{y} \in \R^N $,}
%    \end{align}
    \begin{equation}
      \begin{aligned}
        & 
        \begin{cases} 
            \ds \bigl| \nabla \gamma_\varepsilon(y) \bigr|_{\R^{N}} = \left| \frac{y}{\bigl| [\varepsilon, y] \bigr|_{\mathbb{R}^{N +1}}} \right|_{\R^{N}} \leq \left| \frac{[\varepsilon, y]}{\bigl| [\varepsilon, y] \bigr|_{\mathbb{R}^{N +1}}} \right|_{\R^{N +1}} = 1,
            \\[3ex]
%            \ds \bigl| \nabla \gamma_\varepsilon(y) - \nabla \gamma_{\tilde{\varepsilon}}(\tilde{y}) \bigr|_{\R^N} \leq \left| \frac{[\varepsilon, y]}{\bigl| [\varepsilon, y] \bigr|_{\R^{N +1}}} -\frac{[\tilde{\varepsilon}, \tilde{y}]}{\bigl| [\tilde{\varepsilon}, \tilde{y}] \bigr|_{\R^{N +1}}} \right|_{\R^{N +1}}
%        \\
        \ds 
            |\partial_i \partial_j \gamma_\varepsilon(y)| \leq \frac{1}{\varepsilon}, ~ \mbox{ for all $ \varepsilon > 0 $, $ y \in \R^N $, and $ i, j = 1, \dots, N $.}
            %\qquad \leq \frac{2}{\varepsilon \wedge \tilde{\varepsilon}} \bigl( |\varepsilon -\tilde{\varepsilon}| +|y -\tilde{y}|_{\R^N} \bigr), \mbox{ for all $ \varepsilon, \tilde{\varepsilon} > 0  $ and $ y, \tilde{y} \in \R^N $}.
        \end{cases}
    \end{aligned}\label{exM01}
    \end{equation}
%    \end{subequations}
\end{description}
\end{ex}

\begin{ex}%[Example of Mosco-convergence]
    \label{subdiff2}
    Let $0 \leq \alpha \in H^1(\Omega) \cap L^\infty(\Omega)$ be the fixed function as in (A2), and let $ \{ \gamma_\varepsilon \}_{\varepsilon \geq 0} $ be the sequence of convex functions as in (A3). Then, the following two items hold.
    \begin{description}
      \item[(\,I\,)] Let $\{ \Phi_\varepsilon \}_{\varepsilon \geq 0}$ be a sequence of functionals on $[H]^N$, defined as:
      \begin{equation}
        \Phi_\varepsilon : {\bm w} \in [H]^N \mapsto \Phi_\varepsilon({\bm w}) := \int_\Omega \alpha \gamma_\varepsilon({\bm w}) \,dx \in [0,\infty].
      \end{equation}
      Then, for every $\varepsilon \in [0,\infty)$, $\Phi_\varepsilon$ is proper l.s.c. and convex function, such that
      \begin{equation}
        D(\Phi_\varepsilon) = D(\partial \Phi_\varepsilon) = [H]^N,
      \end{equation}
      and
      \begin{gather}
        \partial \Phi_\varepsilon(\bm{w}) ~ := \left\{ \begin{array}{l}
            \bigl\{ \alpha \nabla \gamma_\varepsilon(\bm{w}) \bigr\}, \mbox{ if $ \varepsilon > 0 $,}
            \\[2ex]
            \left\{ \begin{array}{l|l}
                \alpha \bm{w}^* \in [H]^N & \parbox{2.5cm}{$ \bm{w}^* \in \Sgn(\bm{w}) $ a.e. in $ \Omega $,}
            \end{array} \right\}, \mbox{ if $ \varepsilon = 0 $,}
        \end{array} \right.
        \\
        \mbox{in $[H]^N$, for any $ \bm{w} \in [H]^N $.}
    \end{gather}
    \item[(II)] Let $I \subset (0,T)$ be any open interval, and let $\{ \widehat{\Phi}_\varepsilon^I \}_{\varepsilon \geq 0}$ be a sequence of functionals on $L^2(I;[H]^N) \,( = \bigl[ L^2(I;H) \bigr]^N)$, defined as:
    \begin{equation}
      \widehat{\Phi}_\varepsilon^I : {\bm w} \in L^2(I;[H]^N) \mapsto \widehat{\Phi}_\varepsilon^I ({\bm w}) := \int_I \Phi_\varepsilon({\bm w}(t))\,dt \in [0,\infty].
    \end{equation}
    Then, for every $\varepsilon \in [0,\infty)$, $\widehat{\Phi}_\varepsilon^I$ is proper l.s.c. and convex function, such that
    \begin{equation}
      D(\widehat{\Phi}_\varepsilon^I) = D(\partial \widehat{\Phi}_\varepsilon^I) = L^2(I;[H]^N),
    \end{equation}
    and
    \begin{align}
      \partial \widehat{\Phi}_\varepsilon^I({\bm w}) &= \bigl\{ \tilde{\bm w}^* \in L^2(I;[H]^N) \,|\, \tilde{\bm w}^* (t) \in \partial \Phi_\varepsilon^I({\bm w}(t)) \mbox{ in } [H]^N, \ \mbox{a.e. }t \in I \bigr\}
      \\
      &= \left\{ \begin{array}{l}
        \bigl\{ \alpha \nabla \gamma_\varepsilon(\bm{w}) \bigr\}, \mbox{ if $ \varepsilon > 0 $,}
        \\[2ex]
        \left\{ \begin{array}{l|l}
            \alpha \bm{w}^* \in L^2(I;[H]^N) & \parbox{2.5cm}{$ \bm{w}^* \in \Sgn(\bm{w}) $ a.e. in $ I \times \Omega $,}
        \end{array} \right\}, \mbox{ if $ \varepsilon = 0 $,}
    \end{array} \right.
    \\
    &\mbox{in $L^2(I;[H]^N)$, for any $ \bm{w} \in L^2(I;[H]^N) $.}
    \end{align}
    \end{description}
\end{ex}

\begin{ex}\label{exLaplacian}
  Let $W_0 \subset H^2(\Omega)$ be the closed linear subspace of $H$, as in (A5). Then the operator:
  \begin{equation}
      A_N: \, z \in W_0 \subset H \mapsto A_N z := - \mathit{\Delta} z \in H,
  \end{equation}
  coincides with the subdifferential of the proper, l.s.c., and convex function $\Lambda_N:\, H \longrightarrow [0,\infty]$, defined as:
  \begin{equation}
    \Lambda_N:\, z \in H \mapsto \Lambda_N(z) := \begin{dcases}
      \frac{1}{2} \int_\Omega |\nabla z|^2 dx, & \mbox{if } z \in V, \\
      \infty, & \mbox{otherwise}.
    \end{dcases}
  \end{equation}
    It is known that $ A_N \subset H \times H $ is linear, positive, and self-adjoint, and the domain $ W_0 $ is a Hilbert space, endowed with the inner product:
    \begin{gather}
        (z_1, z_2)_{W_0} := (z_1, z_2)_H +(A_Nz_1, z_2)_H ~ \bigl( = (z_1, z_2)_V \bigr), \mbox{ for $ z_k \in W_0 $, $k = 1, 2$.}
    \end{gather}
    Moreover, there exists a positive constant $ C_0 $ such that:
    \begin{gather}\label{embb01}
        |z|_{H^2(\Omega)}^2 \leq C_0 \bigl( |z|_H^2 +|A_Nz|_H^2 \bigr),~ \mbox{ for all $ z \in W_0 $.}
    \end{gather}
\end{ex}
\medskip

\noindent
\underline{\textbf{\textit{Notations for the time-discretization.}}}
Let $\tau > 0$ be a constant of the time step-size, and let $\{ t_i \}_{i=0}^\infty \subset [0,\infty)$ be the time sequence defined as:
\begin{equation}
  t_i := i\tau,\ i=0,1,2,\ldots.
\end{equation}
Let $X$ be a Banach space. Then, for any sequence $\{ [t_i,z_i] \}_{i=0}^\infty \subset[0,\infty) \times X$, we define the \textit{forward time-interpolation} $[\overline{z}]_\tau \in L^\infty_\mathrm{loc} ([0,\infty); X)$, the \textit{backward time-interpolation} $[\underline{z}]_\tau \in L^\infty_\mathrm{loc} ([0,\infty);X)$ and the \textit{linear time-interpolation} $[z]_\tau \in W^{1,2}_\mathrm{loc} ([0,\infty);X)$, by letting:
\begin{equation}
  \left\{\begin{aligned}
    &[\overline{z}]_\tau(t) := \chi_{(-\infty,0]} z_0 + \sum_{i=1}^\infty \chi_{(t_{i-1}, t_i]}(t) z_i, \\
    &[\underline{z}]_\tau(t) := \sum_{i=0}^\infty \chi_{(t_{i}, t_{i+1}]}(t) z_{i}, \\
    &[z]_\tau (t) := \sum_{i=1}^\infty \chi_{[t_{i-1},t_i)}(t) \left(\frac{t-t_{i-1}}{\tau} z_{i} + \frac{t_i - t}{\tau} z_{i-1}\right),
  \end{aligned}\right. ~{\rm in}~ X,\ {\rm for}~ t \geq 0, \label{eq:time-interpolation}
\end{equation}
respectively.

In the meantime, for any $ q \in [1, \infty) $ and any $ \zeta \in L_\mathrm{loc}^q([0, \infty); X) $, we denote by $ \{ \zeta_i \}_{i = 0}^\infty \subset X $ the sequence of time-discretization data of $ \zeta $, defined as:
\begin{subequations}\label{tI}
\begin{align}\label{tI01}
    & \zeta_0 := 0 \mbox{ in $X$, and }\zeta_i := \frac{1}{\tau} \int_{t_{i -1}}^{t_i} \zeta(\varsigma) \, d \varsigma ~ \mbox{ in $ X $, ~ for $ i = 1, 2, 3, \dots $.}
\end{align}
As is easily checked, the time-interpolations $ [\overline{\zeta}]_\tau, [\underline{\zeta}]_\tau \in L^q_\mathrm{loc}([0, \infty); X) $ for the above $ \{ \zeta_i \}_{i = 0}^\infty $ fulfill that:
\begin{align}\label{tI02}
    & [\overline{\zeta}]_\tau \to \zeta \mbox{ and } [\overline{\zeta}]_\tau \to \zeta \mbox{ in $ L^q_\mathrm{loc}([0, \infty); X) $, as $ \tau \downarrow 0 $.}
\end{align}
\end{subequations}
\medskip

Finally, we mention about a notion of functional convergence, known as ``Mosco-convergence''. 
 
\begin{defn}[Mosco-convergence: cf. \cite{MR0298508}]\label{Def.Mosco}
  Let $ X $ be an abstract Hilbert space. Let $ \Psi : X \rightarrow (-\infty, \infty] $ be a proper, l.s.c., and convex function, and let $ \{ \Psi_n \}_{n = 1}^\infty $ be a sequence of proper, l.s.c., and convex functions $ \Psi_n : X \rightarrow (-\infty, \infty] $, $ n = 1, 2, 3, \dots $.  Then, it is said that $ \Psi_n \to \Psi $ on $ X $, in the sense of Mosco, as $ n \to \infty $, iff. the following two conditions are fulfilled:
  \begin{description}
    \item[(\hypertarget{M_lb}{M1}) The condition of lower-bound:]$ \ds \varliminf_{n \to \infty} \Psi_n(\check{w}_n) \geq \Psi(\check{w}) $, if $ \check{w} \in X $, $ \{ \check{w}_n  \}_{n = 1}^\infty \subset X $, and $ \check{w}_n \to \check{w} $ weakly in $ X $, as $ n \to \infty $. 
    \item[(\hypertarget{M_opt}{M2}) The condition of optimality:]for any $ \hat{w} \in D(\Psi) $, there exists a sequence \linebreak $ \{ \hat{w}_n \}_{n = 1}^\infty  \subset X $ such that $ \hat{w}_n \to \hat{w} $ in $ X $ and $ \Psi_n(\hat{w}_n) \to \Psi(\hat{w}) $, as $ n \to \infty $.
  \end{description}
  As well as, if the sequence of convex functions $ \{ \widehat{\Psi}_\varepsilon \}_{\varepsilon \in \Xi} $ is labeled by a continuous argument $\varepsilon \in \Xi$ with a range $\Xi \subset \mathbb{R}$ , then for any $\varepsilon_{0} \in \Xi$, the Mosco-convergence of $\{ \widehat{\Psi}_\varepsilon \}_{\varepsilon \in \Xi}$, as $\varepsilon \to \varepsilon_{0}$, is defined by those of subsequences $ \{ \widehat{\Psi}_{\varepsilon_n} \}_{n = 1}^\infty $, for all sequences $\{ \varepsilon_n \}_{n=1}^{\infty} \subset \Xi$, satisfying $\varepsilon_{n} \to \varepsilon_{0}$ as $n \to \infty$.
\end{defn}

\begin{rem}\label{Rem.MG}
  Let $ X $, $ \Psi $, and $ \{ \Psi_n \}_{n = 1}^\infty $ be as in Definition~\ref{Def.Mosco}. Then, the following hold.
  \begin{description}
    \item[(\hypertarget{Fact1}{Fact\,1})](cf. \cite[Theorem 3.66]{MR0773850} and \cite[Chapter 2]{Kenmochi81}) Let us assume that
    \begin{equation}\label{Mosco01}
      \Psi_n \to \Psi \mbox{ on $ X $, in the sense of  Mosco, as $ n \to \infty $,}
      \vspace{-1ex}
    \end{equation}
and
\begin{equation*}
\left\{ ~ \parbox{10cm}{
$ [w, w^*] \in X \times X $, ~ $ [w_n, w_n^*] \in \partial \Psi_n $ in $ X \times X $, $ n \in \N $,
\\[1ex]
$ w_n \to w $ in $ X $ and $ w_n^* \to w^* $ weakly in $ X $, as $ n \to \infty $.
} \right.
\end{equation*}
Then, it holds that:
\begin{equation*}
[w, w^*] \in \partial \Psi \mbox{ in $ X \times X $, and } \Psi_n(w_n) \to \Psi(w) \mbox{, as $ n \to \infty $.}
\end{equation*}
    \item[(\hypertarget{Fact2}{Fact\,2})](cf. \cite[Lemma 4.1]{MR3661429} and \cite[Appendix]{MR2096945}) Let $ d \in \mathbb{N} $ denote dimension constant, and let $  S \subset \R^d $ be a bounded open set. Then, under the Mosco-convergence as in \eqref{Mosco01}, a sequence $ \{ \widehat{\Psi}_n^S \}_{n = 1}^\infty $ of proper, l.s.c., and convex functions on $ L^2(S; X) $, defined as:
        \begin{equation*}
            w \in L^2(S; X) \mapsto \widehat{\Psi}_n^S(w) := \left\{ \begin{array}{ll}
                    \multicolumn{2}{l}{\ds \int_S \Psi_n(w(t)) \, dt,}
                    \\[1ex]
                    & \mbox{ if $ \Psi_n(w) \in L^1(S) $,}
                    \\[2.5ex]
                    \infty, & \mbox{ otherwise,}
                \end{array} \right. \mbox{for $ n = 1, 2, 3, \dots $;}
        \end{equation*}
        converges to a proper, l.s.c., and convex function $ \widehat{\Psi}^S $ on $ L^2(S; X) $, defined as:
        \begin{equation*}
            z \in L^2(S; X) \mapsto \widehat{\Psi}^S(z) := \left\{ \begin{array}{ll}
                    \multicolumn{2}{l}{\ds \int_S \Psi(z(t)) \, dt, \mbox{ if $ \Psi(z) \in L^1(S) $,}}
                    \\[2ex]
                    \infty, & \mbox{ otherwise;}
                \end{array} \right. 
        \end{equation*}
        on $ L^2(S; X) $, in the sense of Mosco, as $ n \to \infty $. 
\end{description}
\end{rem}

\begin{ex}[Examples of Mosco-convergence]\label{Rem.ExMG}
    Let $ \varepsilon_0 \geq 0 $ be arbitrary fixed constant, and let $0 \leq \alpha \in H^1(\Omega) \cap L^\infty(\Omega)$ and $\{ \gamma_\varepsilon \}_{\varepsilon \geq 0}$ be as in (A2) and (A3), respectively. Then, the following three items hold.
    \begin{description}
      \item[(O)] $\ds{\gamma_\varepsilon \to \gamma_{\varepsilon_0} \mbox{ on $ \R^N $, in the sense of Mosco, as $ \varepsilon \to \varepsilon_0 $.}}$
      \item[(\,I\,)] Let $\{ \Phi_\varepsilon \}_{\varepsilon \geq 0}$ be the sequence of proper l.s.c. and convex functions on $[H]^N$, as in Example \ref{subdiff2} (I). Then, 
      \begin{equation}
        \Phi_\varepsilon \to \Phi_{\varepsilon_0} \mbox{ on $ [H]^N $, in the sense of Mosco, as $ \varepsilon \to \varepsilon_0 $.}
      \end{equation}
      \item[(II)] Let $I \subset (0,T)$ be an open interval, and let $\{ \widehat{\Phi}_\varepsilon^I \}_{\varepsilon \geq 0}$ be the sequence of proper l.s.c. and convex functions on $L^2(I;[H]^N)$, as a Example \ref{subdiff2} (II). Then, 
      \begin{equation}
        \widehat{\Phi}_\varepsilon^I \to \widehat{\Phi}_{\varepsilon_0}^I \mbox{ on $ L^2(I;[H]^N) $, in the sense of Mosco, as $ \varepsilon \to \varepsilon_0 $.}
      \end{equation}
    \end{description}
\end{ex}

\setcounter{mTh}{0}
\section{Main Theorem}

On the basis of the assumptions and notations as in the previous section, we will set the goal of this paper to prove the following two Main Theorems.

\begin{mTh}[Existence and uniqueness when $\varepsilon > 0$]
    For any $ \varepsilon > 0 $, the relaxed problem (P)$_\varepsilon$ admits a unique solution $ u_\varepsilon : [0,T] \longrightarrow H $, in the following sense.
    \begin{description}
        \item[{\boldmath(S0)$_\varepsilon$}]$ u_\varepsilon \in W^{1, 2}(0, T; W_0)$, and $ u_\varepsilon(0) = u_0 $ in $ H $.
        \item[{\boldmath(S1)$_\varepsilon$}]$ u_\varepsilon $ solves the following %variational inequality:
            %Cauchy problem of evolution equation:
            evolution equation:
%            \begin{align*}\label{eq:def2}
%                \bigl( \partial_t u_\varepsilon(t), \varphi \bigr)_H + \bigl( \alpha \nabla \gamma_\varepsilon & (\nabla u_\varepsilon(t)), \nabla \varphi \bigr)_{[H]^N} + \mu \bigl( \nabla \partial_t u_\varepsilon(t), \nabla \varphi \bigr)_{[H]^N} = 0, 
%                \\
%                &{\rm for~any}~ \varphi \in V, \ {\rm and ~ a.e.}~ t \in (0,T),
%            \end{align*}
%            subject to the initial condition ~$ u_\varepsilon(0) = u_0 $ in $ H $.
            \begin{align}\label{rx:Ev}
                & 
                \partial_t u_\varepsilon(t) -\mathrm{div}\bigl( \alpha \nabla \gamma_\varepsilon(\nabla u_\varepsilon(t)) +\beta \nabla \partial_t u_\varepsilon(t) \bigr) = f(t) \mbox{ in $ H $, a.e. $ t \in (0, T) $,}
            \end{align}
            subject to:
            \begin{align}\label{rx:BC}
                & \nabla u_\varepsilon(t) \cdot n_\Gamma = \nabla \partial_t u_\varepsilon(t) \cdot n_\Gamma = 0 \mbox{ in $ H^{\frac{1}{2}}(\Gamma) $, for a.e. $ t \in (0, T) $.}
            \end{align}
%            and
%            \begin{align}\label{rx:IC}
%                & u_\varepsilon(0) = u_0 \mbox{ in $ H $.}
%            \end{align}
    \end{description}
%  The relaxed problem (P)$_\varepsilon$ has a unique solution $u_\varepsilon$, with the following estimate:
%  \begin{equation}
%    |u_\varepsilon|^2_{W^{1,2}(0,T;V)} \leq C^* (|\alpha|_V^2 + |u_0|^2_{H^2(\Omega)} + 1) \label{eq:main theorem estimate}
%  \end{equation}
%  where $C^*$ is a constant, independent of $\varepsilon$.
\end{mTh}

\begin{mTh}[Existence and uniqueness when $\varepsilon = 0$] 
The problem (P)$_0$, i.e. the problem in case when $ \varepsilon = 0 $, admits a unique solution $ u : [0,T] \longrightarrow H $, in the following sense.
    \begin{description}
        \item[{\boldmath(S0)}]$ u \in W^{1, 2}(0, T; V) \cap L^\infty(0,T; W_0)$, and $ u(0) = u_0 $ in $ H $.
        \item[{\boldmath(S1)}]There exists a function $ \bm{\omega}^* \in \sH $ such that
            \begin{align}\label{org:Rp}
                & \bm{\omega}^*(t) \in \alpha \Sgn(\nabla u(t)) \mbox{ a.e. in $ \Omega $, for a.e. $ t \in (0, T) $,}
            \end{align}
            and $ u $ solves the following evolution equation:
            %variational inequality:
            %Cauchy problem of evolution equation:
            \begin{align}\label{org:Ev}
                & 
                \partial_t u(t) -\mathrm{div}\bigl( \alpha \bm{\omega}^*(t) +\beta \nabla \partial_t u(t) \bigr) = f(t) \mbox{ in $ H $, a.e. $ t \in (0, T) $,}%\bigl( \partial_t u(t), \varphi \bigr)_H + \bigl( \alpha \bm{\omega}^*(t) + \mu \nabla \partial_t u(t), \nabla \varphi \bigr)_{[H]^N} = 0, 
            \end{align}
            subject to:
            \begin{align}\label{org:BC}
                \bigl[\bigl( \alpha \bm{\omega}^*(t)  + \beta \nabla \partial_t u(t) \bigr) \cdot n_\Gamma \bigr]_\Gamma = 0 \mbox{ in $ H^{-\frac{1}{2}}(\Gamma) $, for a.e. $ t \in (0, T) $.}
            \end{align}
    \end{description}
\end{mTh}

    In the Main Theorems, the solutions are obtained as the approximating limits of time-discretization scheme of the relaxed problem (P)$_\varepsilon$. In this light, we take a constant $ \tau > 0 $ of the time-step-size, together with $ \varepsilon > 0 $, and adopt the following time-discretization scheme, as our approximating problem.
\bigskip
\begin{description}
    \item[\textmd{(AP)$_\varepsilon^\tau$:}]to find a time-sequence of functions $ \{ u_{\varepsilon,i} \}_{i = 1}^\infty \subset W_0 $, which fulfills the following time-discretization scheme:
    \begin{align*}
        & \begin{cases}
            \displaystyle \frac{1}{\tau} (u_{\varepsilon,i} -u_{\varepsilon,i-1}) -\mathrm{div} \left( \alpha \nabla \gamma_\varepsilon(\nabla u_{\varepsilon,i}) +\beta \nabla \left( \frac{u_{\varepsilon,i} -u_{\varepsilon,i-1}}{\tau} \right) \right) = f_i ~\mbox{ in $ \Omega $,}
            \\[1ex]
            \nabla u_{\varepsilon,i}|_\Gamma \cdot n_\Gamma = 0 ~\mbox{ on $ \Gamma $, ~ $ i = 1, 2, 3, \dots $,}
        \end{cases}
        \\
        & \hspace{15ex} \mbox{subject to the initial condition $ u_{\varepsilon,0} = u_0 \in W_0 $ in $\Omega$.}
    \end{align*}
\end{description}
%Here, $f_i \in H$ is given by:
%\begin{equation}
%    f_i := \frac{1}{\tau} \int_{t_{i-1}}^{t_i} f(t) \,dt \mbox{ in } X, \ i = 1,2,3,\dots.
%\end{equation}
In view of this, we devote the next section to the study of auxiliary elliptic boundary value problem for the approximating problem (AP)$_\varepsilon^\tau$.

\section{Auxiliary problem}
In this section, we consider the following boundary value problem, denoted by (E):
\begin{equation}
  \mathrm{(E)}\ \ \begin{dcases}
    -\diver \left( \alpha \nabla\gamma_\varepsilon (\nabla u) + \beta \nabla u \right) + \alpha^\circ u = f^\circ \quad {\rm in} \ \Omega, \\[1ex]
    \trace{\nabla u} \cdot n_\Gamma = 0 \quad {\rm on}\ \Gamma.
  \end{dcases}
\end{equation}
In this context, $f^\circ \in L^2(\Omega)$ is a given function, $0 \leq \alpha \in H^1(\Omega) \cap L^\infty(\Omega)$ and $0 < \beta \in W^{1,\infty}(\Omega)$ are fixed functions as in (A2), and $\alpha^\circ: \Omega \longrightarrow (0,\infty)$ belongs to $L^\infty(\Omega)$, satisfying:
\begin{equation}
  \alpha^\circ \in L^\infty(\Omega), \mbox{ and } \essinf \alpha^\circ(\Omega) > 0.
\end{equation}
Additionally, we set:
\begin{equation}
  \delta_0 := \essinf \alpha^\circ (\Omega) \wedge \inf \beta(\Omega) > 0. \label{delta0}
\end{equation}

\begin{defn}
  A function $u:\Omega \longrightarrow \R$ is called a solution to (E), iff. $u \in V$, and:
  \begin{equation}
    \left(\alpha \nabla \gamma_\varepsilon (\nabla u) + \beta \nabla u, \nabla \varphi\right)_{[H]^N} + \left( \alpha^\circ u , \varphi\right)_H = (f^\circ,\varphi)_H, \quad {\rm for \, all}~  \varphi \in V. \label{eq:sol.of.E}
  \end{equation}
\end{defn}

Now we set the goal of this section to prove  the following theorem concerned with strong $H^2$-regularity of solution to (E).

\begin{thm}\label{(E)sol.}
    The problem (E) admits a unique solution $ u \in V $ such that $ u \in W_0 $.
\end{thm}

For the proof of Theorem 1, we take a relaxation argument (constant) $ \kappa > 0  $, and prepare the following relaxed problem, denoted by (E)$_\kappa$.
\begin{equation}
  {\rm (E)}_\kappa \ \begin{dcases}
    \kappa \Lap ^2 u -\diver \left( \alpha \gamma_\varepsilon(\nabla u) + \beta \nabla u \right) + \alpha^\circ u = f^\circ \quad {\rm in} \ \Omega, \\[1ex]
      \trace{\nabla u} \cdot n_\Gamma = 0 ~ \mbox{ and } ~ \trace{\nabla (\Lap u)} \cdot n_\Gamma = 0, ~ \mbox{ on } \Gamma.
  \end{dcases}
\end{equation}

\begin{defn}
    For each $ \kappa > 0 $, a function $u_\kappa: \Omega \longrightarrow \R$ is called a solution to the relaxed problem (E)$_\kappa$, or simply relaxed solution, iff. $u_\kappa \in W_0$, and:
  \begin{equation}
    \begin{aligned}
        & \kappa \left(\Lap u_\kappa, \Lap \varphi\right)_H + \left(\alpha \nabla \gamma_\varepsilon(\nabla u_\kappa) + \beta \nabla u_\kappa, \nabla \varphi\right)_{[H]^N} + \left( \alpha^\circ u_\kappa , \varphi\right)_H
        \\
        & \hspace{14ex} = (f^\circ,\varphi)_H, \mbox{ \ for any}~ \varphi \in W_0.
    \end{aligned}\label{eq:sol.of.Ekappa}
  \end{equation}
\end{defn}

Now, before we deal with Theorem \ref{(E)sol.}, we verify some lemmas concerned with key-properties of relaxed solutions.

\begin{lem}\label{sol. of approx. prob.}
  For each $\kappa > 0$, the problem (E)$_\kappa$ admits a unique solution $u_\kappa \in W_0$, such that $\Lap u_\kappa \in W_0$.
\end{lem}

\begin{proof}
  First, we fix arbitrary $\kappa > 0$, and define a functional $\Upsilon_\kappa:\, H \longrightarrow (-\infty, \infty]$ as follows.
  \begin{equation}
    \Upsilon_\kappa:\, z \in H \mapsto \Upsilon_\kappa(z) := \left\{ \begin{aligned}
      &\frac{\kappa}{2} \int_\Omega |\Lap z|^2 \,dx + \int_{\Omega} \left(\alpha \gamma_\varepsilon(\nabla z) + \frac{\beta}{2} |\nabla z|^2 \right) \,dx \\
      &\quad + \frac{1}{2}\int_\Omega \alpha^\circ |z|^2\,dx - \int_\Omega f^\circ z \,dx \quad {\rm if}~ z \in W_0,\\[1ex]
      &+\infty \quad {\rm otherwise}.
    \end{aligned} \right.
  \end{equation}
  As is easily checked, $\Upsilon_\kappa$ is proper, l.s.c., coercive, and strictly convex on $H$, and the unique minimizer $u_\kappa \in W_0$ solves the relaxed problem (E)$_\kappa$.

  Next, we show $\Lap u_\kappa \in W_0$. To this end, we put:
  \begin{align}
    \left\{ \begin{lgathered}
      v_\kappa := A_N u_\kappa \, (= -\Lap u_\kappa), \mbox{ and}
      \\
      F_\kappa := \diver \bigl( \alpha \nabla \gamma_\varepsilon(\nabla u_\kappa) + \beta \nabla u_\kappa \bigr) - \alpha^\circ u_\kappa + f^\circ + v_\kappa,
    \end{lgathered} \right.\ \mbox{in } H, \label{ken.0707_00}
  \end{align}
  and verify $v_\kappa$ coincides with the unique solution $w_\kappa \in W_0$ to the elliptic boundary value problem:
  \begin{equation}
    \kappa A_N w_\kappa + w_\kappa = F_\kappa \ \mbox{ in } H. \label{ken.0707_01}
  \end{equation}

    Let us take any $z \in H$. Then, by Example \ref{exLaplacian} and Minty's theorem (cf. \cite[Theorem 2.2]{MR2582280} and \cite[Proposition 2.2]{MR0348562}), there exists a (unique) function $\varphi_z \in W_0$ such that
\begin{equation}
  \kappa A_N \varphi_z + \varphi_z = z \ \mbox{ in } H. \label{ken.0707_02}
\end{equation}
    \noeqref{ken.0707_01}
    By using \eqref{eq:sol.of.Ekappa}--\eqref{ken.0707_02}, the coincidence $u_\kappa = w_\kappa \ (\in W_0)$ will be verified as follows.
\begin{align}
  (v_\kappa, z)_H &= \bigl( A_N u_\kappa, \kappa A_N \varphi_z + \varphi_z \bigr)_H = \kappa (\Lap u_\kappa, \Lap \varphi_z)_H - (\Lap u_\kappa, \varphi_z)_H
  \\
  &= -\bigl( \alpha \nabla \gamma_\varepsilon(\nabla u_\kappa) + \beta \nabla u_\kappa, \nabla \varphi_z  \bigr)_{[H]^N} - (\alpha^\circ u_\kappa, \varphi_z)_H + (f^\circ, \varphi_z)_H - (\Lap u_\kappa, \varphi_z)_H
  \\
  &= (F_\kappa, \varphi_z)_H = (-\kappa \Lap w_\kappa + w_\kappa, \varphi_z)_H = (w_\kappa, -\kappa \Lap \varphi_z)_H + (w_\kappa, \varphi_z)_H
  \\
  &= (w_\kappa, \kappa A_N \varphi_z + \varphi_z )_H = (w_\kappa, z)_H, \ \mbox{for any } z \in H.
\end{align}

Thus we conclude this lemma.
\end{proof}

\begin{lem}\label{est99}
  Let us fix $\varepsilon > 0$ and $v \in W_0$. Then, for any $\delta > 0$, there exists a constant $C_1(\delta)$, depending on $\delta$, such that:
  \begin{itemize}
    \item the constant $C_1(\delta)$ is independent of $\varepsilon$ and $v$;
    \item $\bigl( \diver(\alpha \nabla \gamma_\varepsilon(\nabla v)), \Lap v \bigr)_H \geq - \delta|\nabla^2 v|_{[H]^{N \times N}}^2 - C_1(\delta) \bigl( |v|_V^2 + 1 \bigr)$.
  \end{itemize}
\end{lem}
\begin{proof}
  In the light of Lemma \ref{lem:approx.} in Appendix, we derive the following fact:
  \begin{itemize}
    \item[$\sharp$1)] $C^\infty(\overline{\Omega}) \cap W_0$ is dense in $W_0$, in the topology of $H^2(\Omega)$.
  \end{itemize}
  So, to prove this lemma, it is sufficient to verify only in the case when $v \in C^\infty(\overline{\Omega}) \cap W_0$. In this case, we can compute that:
  \begin{align}
    &\left(\diver\left(\alpha \nabla \gamma_\varepsilon(\nabla v) \right), \Lap v\right)_H
    \\
    &\qquad = -\sum_{i=1}^N \sum_{j=1}^N \int_\Omega \alpha [\partial_i \gamma_\varepsilon](\nabla v)\partial_i \left( \partial_j^2 v\right) \,dx 
    \\
    &\qquad = -\sum_{i=1}^N \int_\Omega \alpha [\partial_i \gamma_\varepsilon](\nabla v) \diver\left( \nabla \partial_i v\right) \,dx 
    \\
    &\qquad = \sum_{i=1}^N \int_\Omega \nabla \bigl( \alpha [\partial_i \gamma_\varepsilon](\nabla v)\bigr) \cdot \nabla \partial_i v \, \,dx -\sum_{i=1}^N \int_\Gamma \bigl( \alpha [\partial_i \gamma_\varepsilon](\nabla v)\bigr) \left(\nabla \partial_i v \cdot n_\Gamma \right) \, d\Gamma
    \\
      &\qquad = \sum_{i=1}^N \int_\Omega [\partial_i \gamma_\varepsilon](\nabla v) (\nabla \alpha \cdot \nabla \partial_i v) \,dx + \sum_{i=1}^N \int_\Omega \alpha \nabla \bigl( [\partial_i \gamma_\varepsilon](\nabla v) \bigr) \cdot \nabla \partial_i v \,dx
    \\
    &\qquad \qquad + \sum_{i=1}^N \int_\Gamma \alpha [\partial_i \gamma_\varepsilon](\nabla v) (\nabla v \cdot \partial_i n_\Gamma) \,d\Gamma - \sum_{i=1}^N \int_\Gamma \alpha [\partial_i \gamma_\varepsilon](\nabla v) \partial_i (\nabla v \cdot n_\Gamma) \,d\Gamma
    \\
    &\qquad =: I_1 + I_2 + I_{\Gamma,1} - I_{\Gamma,2}. \label{ken.0708_01}
  \end{align}
  
  Let us take any $\delta > 0$. Then, by using Young's inequality and product-chain rules of differential, the integrals $I_1$ and $I_2$, as in \eqref{ken.0708_01}, are estimated from the below, as follows.
  \begin{align}
    I_1 &= \sum_{i=1}^N \int_\Omega [\partial_i \gamma_\varepsilon](\nabla v) (\nabla \alpha \cdot \nabla \partial_i v) \,dx \geq -\sum_{i=1}^N \int_\Omega |\nabla \alpha \cdot \nabla \partial_i v|\,dx
    \\
    &\geq - \frac{\delta}{2}|\nabla^2 v|_{[H]^{N \times N}}^2 - \frac{N}{2\delta}|\alpha|_V^2, \label{ken.0708_04}
  \end{align}
  and
  \begin{align}
    I_2 &= \sum_{i=1}^N \int_\Omega \alpha \nabla \bigl( [\partial_i \gamma_\varepsilon](\nabla v) \bigr) \cdot \nabla \partial_i v \,dx
    \\
    &= \sum_{i=1}^N \sum_{j=1}^N \sum_{\ell=1}^N \int_\Omega \alpha \bigl([\partial_\ell \partial_i \gamma_\varepsilon](\nabla v) \,\partial_j \partial_\ell v \bigr) \partial_j \partial_i v \, dx \\
    &= \sum_{j=1}^N \int_\Omega \alpha \bigl(\nabla^2 \gamma_\varepsilon(\nabla v) \nabla \partial_j v \bigr)  \cdot \nabla \partial_j v \, dx \geq 0. \label{ken.0708_05}
  \end{align}
  \noeqref{ken.0708_05,ken.0708_04,ken.0708_02}

  Next, let us put:
  \begin{equation}
    |\Gamma| := \H^{N-1}(\Gamma), \mbox{ and } C_\Gamma := |\nabla^2 d_\Gamma|_{[L^\infty(\Gamma)]^{N \times N}},
  \end{equation}
  and apply Lemma \ref{Ehrling} in Appendix to the case when:
  \begin{equation}
    r = r(\delta) := \frac{\delta}{C_\Gamma |\alpha|_{L^\infty(\Omega)}}, \mbox{ with } C_r = C_{r(\delta)} \geq 0.
  \end{equation}
  Then, we can obtain the lower-estimate of $I_{\Gamma,1}$, as follows.
  \begin{align}
    I_{\Gamma,1} &= \sum_{i=1}^N \int_\Gamma \alpha \, [\partial_i \gamma_\varepsilon](\nabla v)\left( \nabla v \cdot \partial_i n_\Gamma  \right) \, d\Gamma 
    \\
    &\geq -\sum_{i=1}^N \int_\Gamma \left| \alpha \, [\partial_i \gamma_\varepsilon](\nabla v) \right| | \nabla v || \partial_i n_\Gamma | \,d\Gamma 
    \\
    &\geq -C_\Gamma |\alpha|_{L^\infty(\Omega)} \int_\Gamma | \nabla v | \, d\Gamma \geq - \frac{C_\Gamma |\alpha|_{L^\infty(\Omega)} }{2} (|\Gamma| + |\nabla v|_{[L^2(\Gamma)]^N}^2)
    \\
    &\geq -\frac{\delta}{2}|\nabla^2 v|_{[H]^{N \times N}}^2 - \frac{C_\Gamma |\alpha|_{L^\infty(\Omega)}}{2} \bigl( C_{r(\delta)} |v|_V^2 + |\Gamma| \bigr). \label{ken.0708_02}
  \end{align}

  As final estimate, we show:% $I_{\Gamma,2} = 0$, by verifying:
  \begin{equation}
      \sum_{i=1}^N [\partial_i \gamma_\varepsilon](\nabla v) \partial_i (\nabla v \cdot n_\Gamma) = 0 \mbox{ a.e. on } \Gamma, ~\mbox{ and hence $ I_{\Gamma, 2} = 0 $.} 
      \label{ken.0708_03}
  \end{equation}
  For simplicity, we put $G := \nabla v \cdot n_\Gamma$. Then, invoking Remark \ref{rem:NS}, the above identify will be equivalently reformulated as follows.
  \begin{align}
      \sum_{i = 1}^N [\partial_i v](\Theta_\ell \Xi_\ell([z', 0 & ])) \, [\partial_i G](\Theta_\ell \Xi_\ell([z', 0]))  = [\nabla v \cdot \nabla G](\Theta_\ell \Xi_\ell ([z',0])) = 0,
    \\
    & \mbox{ for all } z' \in r_\ell \B^{N-1}, \ \mbox{for all } \ell = 1, \dots , M, \label{est132}
  \end{align}
  where $M,\, \{ r_\ell \}_{\ell = 1}^M,\, \{ \Theta_\ell \}_{\ell = 1}^M$, and  $\{ \Xi_\ell \}_{\ell = 1}^M$ are as in Remark \ref{rem:NS}. 

  Since $\Theta_\ell, \, \ell = 1, \dots, M$, are congruent transforms, we can calculate that:
  \begin{align}
    &\quad[\nabla v](\Theta_\ell \Xi_\ell([z', 0])) \cdot [\nabla G](\Theta_\ell \Xi_\ell ([z',0])) \\
    &= \left( (^{t}(D \Xi_\ell))^{-1} \bigl[\nabla (v \circ \Theta_\ell \circ \Xi_\ell)\bigr]([z', 0]) \right) \cdot \left( (^{t}(D \Xi_\ell))^{-1} \bigl[\nabla (G \circ \Theta_\ell \circ \Xi_\ell)\bigr]([z',0]) \right) 
    \\
    &= ^{t}\hspace{-0.7ex}\bigl( \bigl[\nabla (v \circ \Theta_\ell \circ \Xi_\ell)\bigr]([z', 0]) \bigr) \bigl[(D \Xi_\ell)^{-1} (^{t}(D \Xi_\ell))^{-1} \bigr] \bigl[\nabla (G \circ \Theta_\ell \circ \Xi_\ell)\bigr]([z',0]), \label{eq:BI2}
    \\
    &\qquad\qquad\qquad\qquad\quad\  \mbox{ for all } z' \in r_\ell \B^{N-1}, \mbox{ and } \ell = 1, \dots , M.
  \end{align}
  Besides, in the light of the boundary condition $G = \nabla v \cdot n_\Gamma = 0 \mbox{ a.e. on } \Gamma$,
  \begin{align}
    &\left\{ \begin{aligned}
      &\bigl[\partial_i \bigl( G \circ \Theta_\ell \circ \Xi_\ell  \bigr)\bigr]([z',0]) = 0, \ i = 1,\dots, N-1,
      \\
      &\bigl[\partial_N \bigl( v \circ \Theta_\ell \circ \Xi_\ell  \bigr)\bigr]([z',0]) = 0,
    \end{aligned} \right. \label{eq:BI3}
    \\
    &\qquad \qquad \mbox{ for all } z' \in r_\ell \B^{N-1}, \ \ell = 1,\dots,M.
  \end{align}
  Moreover, by virtue of \eqref{D_Xi} in Remark \ref{rem:NS},
  \begin{itemize}
    \item[$\sharp$2)] the $n$-th column of $\ds{\bigl[(D \Xi_\ell)^{-1} (^{t}(D \Xi_\ell))^{-1} \bigr]}$ at $[z',0]$ coincides with ${\bm e}_N = ^t [0,\dots,0,1]$.
  \end{itemize}
  From \eqref{eq:BI2}, \eqref{eq:BI3}, and $\sharp$2), it is inferred that:
  \begin{align}
      {}^{t}\bigl( \bigl[\nabla (v \circ \Theta_\ell \circ \Xi_\ell)\bigr]([z', 0]) & \bigr) \bigl[(D \Xi_\ell)^{-1} (^{t}(D \Xi_\ell))^{-1} \bigr] \bigl[\nabla (G \circ \Theta_\ell \circ \Xi_\ell)\bigr]([z',0]) \label{eq:BI5}
    \\
    &  = 0, ~ \mbox{ for all } z' \in r_\ell \B^{N-1},
  \end{align}
  \eqref{eq:BI2} and \eqref{eq:BI5} are equivalent to \eqref{est132}, and it is sufficient to verify \eqref{ken.0708_03} and $I_{\Gamma,2} = 0$.

  Now, taking into account \eqref{ken.0708_01}--\eqref{ken.0708_03}, we will conclude that the constant
  \begin{equation}
      C_1(\delta) := \frac{1}{2 (\delta \wedge 1)} \bigl( C_\Gamma |\alpha|_{L^\infty(\Omega)} + 1\bigr)\bigl( C_{r(\delta)} + |\Gamma| + N|\alpha|_V^2 \bigr)
  \end{equation}
  will be the required constant. 
  
  Thus, we finish the proof of this lemma.
  
\end{proof}

\begin{proof}[The proof of Theorem \ref{(E)sol.}]
  First, we note that the problem (E) is equivalent to the minimization problem for the following proper l.s.c. convex function $\Upsilon:\, H \longrightarrow (-\infty, \infty]$, defined as:
\begin{equation}
  \Upsilon:\,z \in H \mapsto \Upsilon(z) := \left\{ \begin{aligned}
    &\int_{\Omega} \left(\alpha \gamma_\varepsilon(\nabla z) + \frac{\beta}{2} |\nabla z|^2 \right) \,dx \\
    &\quad + \frac{1}{2}\int_\Omega \alpha^\circ |z|^2\,dx - \int_\Omega f^\circ z \,dx, \quad {\rm if}~ z \in V,\\
    &+\infty, \quad {\rm otherwise}.
  \end{aligned} \right. \label{defOfPhi}
\end{equation}
In the light of \eqref{delta0} and \eqref{defOfPhi}, $\Upsilon$ is coercive and strictly convex on $H$. So, $\Upsilon$ admits a unique minimizer, and it directly leads to the existence and uniqueness of solution $u \in V$ to the problem (E).

Hence, our remaining task is to show $u \in W_0$. To this end, we put $\varphi = u_\kappa$ in \eqref{eq:sol.of.Ekappa}, and derive that:
\begin{equation}
  \frac{\delta_0}{2} |u_\kappa|_{V}^2 \leq \frac{1}{2 \delta_0} |f^\circ|_{H}^2, \mbox{ for any } \kappa > 0, \label{est_V}
\end{equation}
via the following computations:
\begin{gather}
  (\alpha \nabla \gamma_\varepsilon(\nabla u_\kappa), \nabla u_\kappa)_H = \int_\Omega \alpha \frac{|\nabla u_\kappa|^2}{\sqrt{\varepsilon^2 + |\nabla u_\kappa|^2}} \,dx \geq 0, \\
  (\beta \nabla u_\kappa, \nabla u_\kappa)_H \geq \delta_0 |\nabla u_\kappa |^2_{[H]^N}, \\
  (\alpha^\circ u_\kappa, u_\kappa)_H \geq \delta_0 |u_\kappa|_H^2,
\end{gather}
and
\begin{equation}
  (f^\circ ,u_\kappa)_H \leq \frac{\delta_0}{2}| u_\kappa|_{H}^2 + \frac{1}{2\delta_0} |f^\circ|_H^2, \mbox{ for any } \kappa > 0.
\end{equation}
Besides, by putting $\varphi = A_N u_\kappa$ in \eqref{eq:sol.of.Ekappa}, it follows from \eqref{delta0} and $A_N u_\kappa = - \Lap u_\kappa \in W_0$ that
\begin{gather}
  \kappa |\nabla A_N u_\kappa|_{[H]^N}^2 + \delta_0 |A_N u_\kappa|_H^2\leq \bigl( \alpha \nabla \gamma_\varepsilon(\nabla u_\kappa), \nabla \Lap u_\kappa \bigr)_{[H]^N} 
  \\
  + ( \nabla \beta \cdot \nabla u_{\kappa} -\alpha^\circ u_\kappa + f^\circ, A_N u_\kappa )_H, \mbox{ for any } \kappa > 0.
\end{gather}
So, by using Young's inequality, the above inequality is reduced to:
\begin{gather}
    \frac{\delta_0}{2} |A_N u_\kappa|_H^2 \leq \bigl( \alpha \nabla \gamma_\varepsilon (\nabla u_\kappa), \nabla \Lap u_\kappa \bigr)_{[H]^N}
    \\
    +\frac{2}{\delta_0} \bigl( (|\alpha^\circ|_{L^\infty(\Omega)}^2 +|\nabla \beta|_{[L^\infty(\Omega)]^N}^2) |u_\kappa|_V^2 + |f^\circ|_H^2 \bigr), \mbox{ for any } \kappa > 0.
\label{est_H2}
\end{gather}
Here, let us take $\delta > 0$, arbitrary. Then, by Lemma \ref{est99}, there exists a constant $C_1 = C_1(\delta)$, depending on $\delta$ and $\alpha$, such that $C_1 = C_1(\delta)$ is independent of $\kappa > 0$, and
\begin{align}
  &\bigl( \alpha \nabla \gamma_\varepsilon(\nabla u_\kappa), \nabla \Lap u_\kappa \bigr)_{[H]^N} = - \bigl( \diver(\alpha \nabla \gamma_\varepsilon(\nabla u_\kappa)), \Lap u_\kappa \bigr)
  \\
  &\quad \leq \delta|\nabla^2 u_\kappa|_{[H]^{N \times N}}^2 + C_1(\delta) (|u_\kappa|_V^2 + 1), \mbox{ for any } \kappa > 0. \label{est_H20}
\end{align}
    Additionally, by using \eqref{embb01} and \eqref{est_V}, the estimate \eqref{est_H20} can be continued as follows:
\begin{align}
  &\bigl( \alpha \nabla \gamma_\varepsilon(\nabla u_\kappa), \nabla \Lap u_\kappa \bigr)_{[H]^N}
  \\
  &\quad \leq C_0\delta \bigl( |A_N u_\kappa|_H^2 + |u_\kappa|_H^2 \bigr) + C_1(\delta) \left( \frac{1}{\delta_0^2}|f^\circ|_H^2 + 1 \right), \mbox{ for any } \kappa > 0. \label{est_H21}
\end{align}
Based on these, we set:
\begin{equation}
  \delta = \delta_1 := \frac{\delta_0}{4C_0}. \label{est_H22}
\end{equation}
Then, from \eqref{est_H2}, \eqref{est_H21}, and \eqref{est_H22}, it is deduced that:
    \begin{gather}
        \frac{\delta_0}{4}|A_N u_\kappa|_H^2 \leq \left( \frac{\delta_0}{4} + \frac{2}{\delta_0} (|\alpha^\circ|_{L^\infty(\Omega)}^2 + |\nabla \beta|_{[L^\infty(\Omega)]^N}^2) \right) |u_\kappa|_H^2 
        \\
        + 2\bigl( C_1(\delta_1) + 1 \bigr)\left( \frac{1}{\delta_0^2 \wedge 1}|f^\circ|_H^2 + 1 \right), \mbox{ for any }\kappa > 0. \label{est_H23}
\end{gather}
Taking a sum of \eqref{est_V} and \eqref{est_H23}, and using \eqref{embb01}, we will find a constant $C_2$, independent of $\kappa > 0$, such that
\begin{gather}
  |u_\kappa|_{H^2(\Omega)}^2 \leq C_2 \bigl( |f^\circ|_H^2 + 1 \bigr),\mbox{ for any } \kappa > 0, \label{est_H24}
    \\
    \mbox{where} ~~
    C_2 := \frac{8C_0}{(\delta_0^4 \wedge 1)} \left( C_1(\delta_1) + |\alpha^\circ |_{L^\infty(\Omega)}^2 + |\nabla \beta|_{[L^\infty(\Omega)]^N}^2 + 2 \right).
\end{gather}

On account of \eqref{est_H24}, there exist a subsequence of $\{ u_\kappa \}_{\kappa > 0} \subset W_0$ (not relabeled), and the function $\tilde u \in W_0$, such that:
\begin{equation}
  u_{\kappa} \to \tilde u \mbox{ in } V, \mbox{ and weakly in } H^2(\Omega), \mbox{ as } \kappa \downarrow 0.
\end{equation}
Therefore, letting $\kappa \downarrow 0$ in \eqref{eq:sol.of.Ekappa}, we can observe that the limit $\tilde u \in W_0$ is a solution to (E). Due to the uniqueness of solution, any solution $u$ to (E) should coincide with $\tilde u \in W_0$.
\medskip

Thus, we conclude Theorem 1.
\end{proof}

\section{Proofs of Main Theorems}
This section is devoted to the proofs of Main Theorems 1 and 2. 

First, referring to Theorem \ref{(E)sol.}, the approximating problem (AP)$^\tau_\varepsilon$ admits a unique solution, i.e., for any $ \varepsilon > 0 $, there exists a sequence of functions $\{ u_{\varepsilon,i} \}_{i = 1}^\infty \subset W_0$, such that:
\begin{gather}
  \frac{1}{\tau}(u_{\varepsilon,i} - u_{\varepsilon,i-1}, \varphi)_H + \left( \alpha \nabla \gamma_\varepsilon(\nabla u_{\varepsilon,i}),\nabla \varphi \right)_{[H]^N} + \frac{1}{\tau} \left(\beta\nabla (u_{\varepsilon,i} - u_{\varepsilon,i-1}), \nabla \varphi \right)_{[H]^N}
  \\
  = (f_i, \varphi)_H,\ \mbox{ for any } \varphi \in V ,\ i = 1,2,\ldots, \mbox{ with } u_{\varepsilon,0} = u_0 \mbox{ in } H.
 \label{AP.sol} 
\end{gather}
\noeqref{tI01,tI02}
Here, we set $n_\tau := \min \{ n \in \N \,|\, n\tau \geq T \}$, $ T_\tau := \tau \cdot n_\tau $, and $\Delta_{i,\tau} := [t_{i-1}, t_i)$. Besides, in the light of \eqref{tI}, we let $\tau_0 \in (0, \frac{1}{2})$ be sufficiently small such that:
\begin{equation}\label{fest}
    \tau \sum_{i = 1}^{n_\tau} |f_i|_H^2 = \bigl| [\overline{f}]_\tau \bigr|_{L^2(0, T_\tau;H)}^2 \leq |f|_{L^2(0,T;H)}^2 + 1, \mbox{ for all } \tau \in (0,\tau_0).
\end{equation}

Now, we will see estimates of the approximating solution $\{ u_{\varepsilon,i} \}_{i = 1}^{n_\tau} \subset W_0 $.\begin{lem}\label{Lem1}
    Let $\tau \in (0,\tau_0)$ and $\varepsilon > 0$. Then, it holds that:
  \begin{gather}
    |u_{\varepsilon,i}|_V^2 \leq C_3 (|u_0|_V^2 + |f|_{L^2(0,T;H)}^2 + 1), \ i = 1,2,3, \dots, n_\tau, \label{eq:Lem1.1}
    \\
    \frac{1}{\tau} \sum_{i=1}^{n_\tau} \bigl( |u_{\varepsilon,i} - u_{\varepsilon,i-1}|_H^2 + \delta_* | \nabla (u_{\varepsilon,i} - u_{\varepsilon,i-1})|_{[H]^N}^2 \bigr) \leq C_4 (|u_0|_V^2 + |f|_{L^2(0,T;H)}^2 + \varepsilon^2 + 1) \label{eq:Lem1.2},
  \end{gather}
  where
  \begin{align}
      C_3 := \frac{2(|\beta|_{L^\infty(\Omega)} + 1)}{\delta_* \wedge 1} e^{2T}, \mbox{ and } C_4 := |\alpha|_H^2 + \L^N(\Omega) + 1.
  \end{align}
  \begin{proof}
    By multiplying the both sides of (AP)$^\tau_\varepsilon$ with $u_{\varepsilon,i}$, we obtain the following estimates:
    \begin{align}
    \frac{1}{2\tau} \left( |u_{\varepsilon,i}|^2_H - |u_{\varepsilon,i-1}|_H^2 \right) + \frac{1}{2\tau} \int_\Omega \beta\left( |\nabla u_{\varepsilon, i} |^2 - |\nabla u_{\varepsilon,i-1}|^2 \right)\,dx \label{gronwall1}
    \\
    \leq \frac{1}{2}|u_{\varepsilon,i}|_H^2 + \frac{1}{2} |f_i|_H^2, \ \mbox{for every }i = 1,2,\ldots, n_\tau,
  \end{align}
  via following calculation:
  \begin{gather}
    \frac{1}{\tau} \left( u_{\varepsilon,i}- u_{\varepsilon,i-1}, u_{\varepsilon,i} \right)_H \geq \frac{1}{2\tau} \left( |u_{\varepsilon,i}|^2_H - |u_{\varepsilon,i-1}|_H^2 \right),
    \\
    \left( \alpha \nabla \gamma_\varepsilon(\nabla u_{\varepsilon,i}), \nabla u_{\varepsilon,i} \right)_{[H]^N} = \int_\Omega \alpha \frac{|\nabla u_{\varepsilon,i}|^2}{\sqrt{\varepsilon^2 + |\nabla u_{\varepsilon,i}|^2}} dx \geq 0, \label{eq:H1-bdd1}
    \\
    \frac{1}{\tau} \left( \beta \nabla (u_{\varepsilon,i}- u_{\varepsilon,i-1}), \nabla u_{\varepsilon,i} \right)_{[H]^N} \geq \frac{1}{2\tau} \int_\Omega \beta\left( |\nabla u_{\varepsilon, i} |^2 - |\nabla u_{\varepsilon,i-1}|^2 \right)\,dx,
    \\
    (f_i, u_{\varepsilon,i})_H \leq \frac{1}{2}|f_i|_H^2 + \frac{1}{2} |u_{\varepsilon,i}|_H^2, \mbox{ for } i = 1,2,3,\dots, n_\tau.
  \end{gather}
  Hence, we obtain that
\begin{equation}
  \frac{1}{\tau} (X_i - X_{i-1}) \leq (X_i + |f_i|_H^2),\label{gronwall.Lem}
\end{equation}
with
\begin{equation}
  X_i := |u_{\varepsilon,i}|_H^2 + |\sqrt{\beta} \nabla u_{\varepsilon,i}|_{[H]^N}^2, \ \mbox{for } i = 1,2,3,\dots, n_\tau.
\end{equation}
      Here, with \eqref{fest} and $ 0 < \tau < \tau_0 < \frac{1}{2} $ in mind, we can apply the discrete version of Gronwall's lemma (cf. \cite[Section 3.1]{emmrich1999discrete}) to \eqref{gronwall.Lem}, and derive that:
\begin{equation}
  X_i \leq e^{2T} \bigl( X_0 + |f|_\sH^2 + 1 \bigr), \ \mbox{for } i = 1,2,3,\dots, n_\tau.
\end{equation}
This implies that:
\begin{equation}
  |u_{\varepsilon,i}|_V^2 \leq \frac{(|\beta|_{L^\infty(\Omega)} + 1)}{\delta_* \wedge 1}e^{2T} \bigl( |u_0|_V^2 + |f|_\sH^2 + 1 \bigr), \ \mbox{for } i = 1,2,3, \dots, n_\tau. \label{073101}
\end{equation}
In addition, by multiplying the both sides with $u_{\varepsilon,i} - u_{\varepsilon,i-1}$ and using the convexity of $\gamma_\varepsilon$, it is observed that:
\begin{gather}
  \frac{1}{2\tau}|u_{\varepsilon,i} - u_{\varepsilon,i-1}|_H^2 + \frac{\delta_*}{\tau}|\nabla (u_{\varepsilon,i} - u_{\varepsilon,i-1})|_{[H]^N}^2 + \int_\Omega \alpha \gamma_\varepsilon(\nabla u_{\varepsilon,i}) \,dx \label{gronwall2}
  \\
  \leq \int_\Omega \alpha \gamma_\varepsilon (\nabla u_{\varepsilon,i-1}) \,dx + \frac{\tau}{2}|f_i|_H^2, \quad \mbox{for }i = 1,2,3,\dots, n_\tau,
\end{gather}
via following computation:
\begin{gather}
  \frac{1}{\tau} (u_{\varepsilon,i} - u_{\varepsilon,i-1}, u_{\varepsilon,i} - u_{\varepsilon,i-1})_H = \frac{1}{\tau}|u_{\varepsilon,i} - u_{\varepsilon,i-1}|_H^2,\
  \\
  \left( \alpha \nabla\gamma_\varepsilon(\nabla u_{\varepsilon,i}), \nabla u_{\varepsilon,i} - \nabla u_{\varepsilon,i-1} \right)_{[H]^N} \geq \int_\Omega \alpha \gamma_\varepsilon(\nabla u_{\varepsilon,i}) dx - \int_\Omega \alpha \gamma_\varepsilon(\nabla u_{\varepsilon,i-1}) dx,
  \\
  \frac{1}{\tau}\left( \beta\nabla(u_{\varepsilon,i} - u_{\varepsilon,i-1}) , \nabla (u_{\varepsilon,i} - u_{\varepsilon,i-1})\right)_{[H]^N} \geq \frac{\delta_*}{\tau}\left| \nabla u_{\varepsilon,i} - \nabla u_{\varepsilon,i-1} \right|_{[H]^N}^2,
  \\
    (f_i, u_{\varepsilon,i} - u_{\varepsilon,i-1})_H \leq \frac{1}{2\tau} |u_{\varepsilon,i} - u_{\varepsilon,i-1}|_H^2 + \frac{\tau}{2}|f_i|_H^2, \mbox{ for $ i = 1,2,3, \dots, n_\tau $.}
\end{gather}
      Taking a sum of the inequalities in \eqref{gronwall2}, and using \eqref{fest}, we obtain that:
\begin{align}
  &\frac{1}{\tau}\sum_{i=1}^{n_\tau} \bigl(|u_{\varepsilon,i} - u_{\varepsilon,i-1}|_H^2 + \delta_*|\nabla (u_{\varepsilon,i} - u_{\varepsilon,i-1})|_{[H]^N}^2\bigr) \label{073102}
  \\
  &\qquad \leq 2\int_\Omega \alpha \gamma_\varepsilon(\nabla u_0) \,dx + |f|_\sH^2 + 1 
  \\
  &\qquad \leq |\alpha|_H^2 + \varepsilon^2 \L^N(\Omega) + |u_0|_V^2 + |f|_\sH^2 + 1.
  \\
  &\qquad \leq \bigl( |\alpha|_H^2 + \L^N(\Omega) + 1 \bigr)\bigl( |u_0|_V^2 + |f|_{\sH}^2 + \varepsilon^2 + 1 \bigr).
\end{align}
\eqref{073101} and \eqref{073102} finish the proof of this lemma.
  \end{proof}
\end{lem}

\begin{lem} \label{Lem2}
  There exists a constant $C_5 > 0$, independent of $\varepsilon $ and $ i$, such that
  \begin{equation}
    |u_{\varepsilon,i}|_{H^2(\Omega)}^2 \leq C_5 \bigl( |u_0|_{H^2(\Omega)}^2 + |f|_{\sH}^2 + \varepsilon^2 +  1 \bigr).
  \end{equation}
  \begin{proof}
    Let us consider to multiply the both sides of (AP)$^\tau_\varepsilon$ by $A_N u_{\varepsilon,i}\ ( = -\Lap u_{\varepsilon,i})$. 
    
    First, by Young's inequality, we easily have:
      \begin{gather}
        \frac{1}{\tau}(u_{\varepsilon,i} - u_{\varepsilon,i-1}, A_N u_{\varepsilon,i})_H \geq \frac{1}{2\tau} \bigl( |\nabla u_{\varepsilon,i}|_{[H]^N}^2 - |\nabla u_{\varepsilon,i-1}|_{[H]^N}^2 \bigr), \label{gronwall3}
      \end{gather}
      and
      \begin{gather}
          (f_i, A_N u_{\varepsilon,i})_H \leq \frac{1}{2}|f_i|_H^2 + \frac{1}{2}|A_N u_{\varepsilon,i}|_H^2 \leq \frac{1}{2}|f_i|_H^2 + \frac{1}{2\delta_*}\int_\Omega \beta |\Lap u_{\varepsilon,i}|^2 \,dx, 
          \\
          \mbox{for $ i = 1,2 , 3, \dots, n_\tau $.}
      \end{gather}
Secondly, we take the embedding constant $C_0$ as in \eqref{embb01}, and apply Lemma \ref{est99} to the case when
\begin{equation}
    \delta = \frac{1}{2C_0}, \mbox{ with the corresponding constant } \tilde C_1 := C_1(\delta) \, \left( = {\textstyle C_1 \bigl(\frac{1}{2C_0} \bigr)} \right).
\end{equation}
      Then, from \eqref{fest} and (A2), it is deduced that:
\begin{align}
  &\bigl( -\diver(\alpha \nabla \gamma_\varepsilon( \nabla u_{\varepsilon,i} )), A_N u_{\varepsilon,i} \bigr)_H = \bigl( \diver(\alpha \nabla \gamma_\varepsilon( \nabla u_{\varepsilon,i} )), \Lap u_{\varepsilon,i} \bigr)_H \label{gronwall4}
  \\
  &\qquad\geq - \frac{1}{2C_0}| \nabla^2 u_{\varepsilon,i} |_{[H]^{N \times N}}^2 - \tilde C_1(|u_{\varepsilon,i}|_V^2 + 1)
  \\
  &\qquad\geq - \frac{1}{2}(|A_N u_{\varepsilon,i}|_H^2 + |u_{\varepsilon,i}|_H^2) - \tilde C_1(|u_{\varepsilon,i}|_V^2 + 1)
  \\
  &\qquad\geq -\frac{1}{2\delta_*} \int_\Omega \beta|\Lap u_{\varepsilon,i}|^2 \,dx - C_3\bigl( \tilde C_1 + 1 \bigr)(|u_0|_V^2 + |f|_{L^2(0,T;H)}^2 + 1),
    \\
    & \qquad\qquad\mbox{for $ i = 1, 2, 3, \dots, n_\tau $.}
\end{align}
Finally, by using Young's inequality, one can compute that:
\begin{align}
  &\frac{1}{\tau} \bigl( -\diver(\beta \nabla (u_{\varepsilon,i} - u_{\varepsilon,i-1})), A_N u_{\varepsilon,i} \bigr)_H \label{gronwall5}
  \\
  &\quad \geq \frac{1}{2\tau} \int_\Omega \beta \bigl( |\Lap u_{\varepsilon,i}|^2 - |\Lap u_{\varepsilon,i-1}|^2 \bigr)\,dx + \frac{1}{\tau} \bigl( \nabla \beta \cdot \nabla (u_{\varepsilon,i} - u_{\varepsilon,i-1}), \Lap u_{\varepsilon,i} \bigr)_H,
\end{align}
and furthermore,
\begin{align}
  &\frac{1}{\tau}(\nabla \beta \cdot \nabla(u_{\varepsilon, i} - u_{\varepsilon, i-1}), \Lap u_{\varepsilon,i} )_H \label{gronwall6}
  \\
  &\quad\geq - \frac{1}{\tau}\int_\Omega |\nabla \beta| \bigl| \nabla (u_{\varepsilon,i} - u_{\varepsilon,i-1}) \bigr| |\Lap  u_{\varepsilon,i}| \,dx 
  \\
  &\quad\geq - \frac{1}{2\delta_*}|\nabla \beta|_{[L^\infty(\Omega)]^N}^2 |\Lap  u_{\varepsilon,i}|_H^2 -\frac{\delta_*}{2\tau^2}|\nabla(u_{\varepsilon, i} - u_{\varepsilon,i-1})|_{[H]^N}^2
  \\
  &\quad\geq - \frac{1}{2\delta_*^2}|\nabla \beta|_{[L^\infty(\Omega)]^N}^2 \int_\Omega \beta |\Lap u_{\varepsilon,i}|^2\,dx 
  \\
  &\quad\qquad + \frac{1}{2\tau} \left( \int_\Omega \alpha \gamma_\varepsilon(\nabla u_{\varepsilon, i}) \, dx - \int_\Omega \alpha \gamma_\varepsilon(\nabla u_{\varepsilon,i-1})\, dx \right) - \frac{1}{4}|f_i|_H^2,
    \\
    &\mbox{for $ i = 1, 2, 3, \dots, n_\tau $.}
\end{align}\noeqref{gronwall4, gronwall5} 

On account of \eqref{gronwall3}--\eqref{gronwall6}, we will infer that:
\begin{equation}
    \frac{1}{\tau} (Y_i - Y_{i-1}) \leq \tilde C_* (Y_i + |f_i|_H^2 + 1), \label{gronwall10}
  \end{equation} 
where
\begin{equation}
    Y_i := |\nabla u_{\varepsilon,i}|_{[H]^N}^2 + \int_\Omega \beta |\Lap  u_{\varepsilon,i}|^2 \,dx + \int_\Omega \alpha \gamma_\varepsilon (\nabla u_{\varepsilon,i} ) \,dx, \mbox{ for $ i = 1,2,3 \ldots, n_\tau $,}
\end{equation}
and
\begin{equation}
  \tilde C_* := \frac{2C_3}{\delta_*^2 \wedge 1}\bigl(|\nabla \beta|_{[L^\infty(\Omega)]^N}^2 + 1 \bigr) \bigl( \tilde C_1 + 1 \bigr) \bigl( |u_0|_V^2 + |f|_{\sH}^2 + 1 \bigr) \geq 2.
\end{equation}
Here, let us take $\tau_*$ to satisfy:
\begin{equation}
  \tau_* < \min \left\{ \tau_0, \frac{1}{2\tilde C_*} \right\}, ~{\rm and~in~particular,}~ 1-\tau_* \tilde C_* > \frac{1}{2}.
\end{equation}
Then, applying the discrete version of Gronwall's lemma (cf. \cite[Section 3.1]{emmrich1999discrete}) to \eqref{gronwall10}, it is observed that:
\begin{align}
    Y_i \leq \tilde C_* e^{2\tilde C_* T} (Y_0 + |f|_\sH^2 + T + 1), ~~ \mbox{for $ i = 1, 2, 3,\ldots, n_\tau $.}
\end{align}
Having in mind \eqref{embb01}, \eqref{fest}, and Lemma \ref{Lem1}, we arrive at:
\begin{align}
  &\quad |u_{\varepsilon,i}|_{H^2(\Omega)}^2 \leq C_0 \bigl( | A_N u_{\varepsilon,i} |_H^2 + |u_{\varepsilon,i}|_H^2  \bigr) = C_0(|\Lap u_{\varepsilon,i}|_H^2 + |u_{\varepsilon,i}|_H^2 ) 
  \\
  &\leq \frac{C_0}{\delta_* \wedge 1} Y_i + C_0 C_3(|u_0|_V^2 + |f|_{\sH}^2 + 1)
  \\
  &\leq \frac{C_0 \tilde C_* e^{2 \tilde C_* T}}{\delta_* \wedge 1} (Y_0 + |f|_\sH^2 + T + 1) + C_0 C_3(|u_0|_V^2 + |f|_{\sH}^2 + 1)
  \\
  &\leq \frac{C_0 \tilde C_* e^{2 \tilde C_* T}}{\delta_* \wedge 1} \left( (2 + N|\beta|_{L^\infty(\Omega)}) |u_0|_{H^2(\Omega)}^2 + |\alpha|_H^2 + \varepsilon^2 \L^N(\Omega) + |f|_{\sH}^2 + T + 1\right)
  \\
  &\quad\qquad + C_0 C_3 \bigl( |u_0|_V^2 + |f|_{\sH}^2 + 1 \bigr)
  \\
    &\leq C_5 \bigl( |u_0|_{H^2(\Omega)}^2 + |f|_{\sH}^2 + \varepsilon^2 + 1 \bigr), ~ \mbox{for $ i = 1, 2, 3, \dots, n_\tau $,}
  %\\
  %&\leq \frac{2 C_0 C_3 \tilde C_* e^{2\tilde C_* T}}{1 \wedge \delta_*} \left(N|\beta|_{L^\infty(\Omega)} + |\alpha|_H^2 + \L^N(\Omega) + T  + 2 \right)\bigl( |u_0|_{H^2(\Omega)}^2 + |f|_{\sH}^2 + \varepsilon^2 + 1 \bigr),
%  \\
%  &\hspace{6.5cm} \mbox{ for } i = 1,2,3,\dots, n_\tau.
\end{align}
where
\begin{equation}
  C_5 := \frac{2 C_0 C_3 \tilde C_* e^{2\tilde C_* T}}{\delta_* \wedge 1} \left(N|\beta|_{L^\infty(\Omega)} + |\alpha|_H^2 + \L^N(\Omega) + T  + 2 \right).
\end{equation}

Thus we conclude this lemma.
\end{proof}
\end{lem}

\begin{lem} \label{Lem3}
    There exists a constant $C_6 > 0$, independent of $\varepsilon $ and $ i$, such that
  \begin{equation}
    \frac{1}{\tau} \sum_{i=1}^{n_\tau} |u_{\varepsilon,i} - u_{\varepsilon,i-1}|_{H^2(\Omega)}^2 \leq \frac{C_6}{\varepsilon^2 \wedge 1} \left( |u_0|_{H^2(\Omega)}^2 + |f|_\sH^2 + \varepsilon^2 + 1 \right).
  \end{equation}
  \begin{proof}
    
Let us consider to take the inner product of (AP)$_\varepsilon^\tau$ with $A_N (u_{\varepsilon, i} - u_{\varepsilon,i-1}) \ (= -\Lap (u_{\varepsilon, i} - u_{\varepsilon,i-1}))$. 

First, by Green's formula, we immediately have:
  \begin{gather}
    \left(\frac{u_{\varepsilon,i} - u_{\varepsilon,i-1}}{\tau}, A_N (u_{\varepsilon,i} - u_{\varepsilon,i-1})\right)_H = \frac{1}{\tau} |\nabla (u_{\varepsilon,i} - u_{\varepsilon,i-1})|_{[H]^N}^2 \geq 0, \label{eq:TD.H2.1}
    \\
    \mbox{for  $ i = 1,2,3 ,\dots, n_\tau $. }
  \end{gather}
      Secondly, from \eqref{exM01} in Example \ref{exConvex}, it is estimated that:
\begin{align}
      \bigl| \diver\left( \alpha \nabla \gamma_\varepsilon(\nabla u_{\varepsilon, i}) \right) \bigr|_H^2 & = \left| \nabla \alpha \cdot \nabla \gamma_\varepsilon(\nabla u_{\varepsilon, i}) + \alpha \sum_{k,j=1}^N [\partial_k \partial_j \gamma_\varepsilon](\nabla u_{\varepsilon,i}) \, \partial_j \partial_k u_{\varepsilon,i} \right|_H^2 \\
    & \leq (N^2 + 1) \left( |\alpha|_V^2 + \frac{|\alpha|_{L^\infty(\Omega)}^2}{\varepsilon^2} |u_{\varepsilon,i}|_{H^2(\Omega)}^2 \right),
\end{align}
and therefore,
\begin{align}
    &\left( -\diver\left( \alpha \nabla \gamma_\varepsilon(\nabla u_{\varepsilon,i}) \right), A_N (u_{\varepsilon,i} - u_{\varepsilon,i-1}) \right)_H %= \left( \diver\left( \alpha \nabla \gamma_\varepsilon(\nabla u_{\varepsilon,i}) \right), \Lap (u_{\varepsilon,i} - u_{\varepsilon,i-1}) \right)_H 
  \\
  & \quad \geq -\frac{\delta_*}{4\tau} |A_N (u_{\varepsilon,i} - u_{\varepsilon,i-1}) |_H^2 - \frac{\tau(N^2 + 1)}{\delta_*}\left( |\alpha|_V^2 + \frac{|\alpha|_{L^\infty(\Omega)}^2}{\varepsilon^2} |u_{\varepsilon,i}|_{H^2(\Omega)}^2 \right),
    \label{eq:TD.H2.2}
  \\
  & \quad \qquad \mbox{for } i = 1,2,3, \dots, n_\tau.
\end{align}
Finally, by using Young's inequality, we can compute that:
\begin{align}
  &\left( -\diver\left( \frac{\beta}{\tau} \nabla (u_{\varepsilon, i} - u_{\varepsilon,i-1}) \right), A_N (u_{\varepsilon, i} - u_{\varepsilon,i-1}) \right)_H 
    \\
    = & \frac{1}{\tau} \left( \beta \Lap (u_{\varepsilon,i} - u_{\varepsilon,i-1}), \Lap (u_{\varepsilon,i} - u_{\varepsilon,i-1}) \right)_H 
    \\
    & \qquad + \frac{1}{\tau} \left( \nabla \beta \cdot \nabla(u_{\varepsilon,i} - u_{\varepsilon,i-1}), \Lap (u_{\varepsilon}- u_{\varepsilon,i-1}) \right)_H\\
    \geq & \frac{3\delta_*}{4\tau} |A_N (u_{\varepsilon,i}- u_{\varepsilon,i-1})|_H^2 - \frac{|\nabla \beta|_{[L^{\infty}(\Omega)]^N}^2}{\delta_* \tau} |\nabla (u_{\varepsilon, i} - u_{\varepsilon,i-1}) |_{[H]^N}^2,
\label{eq:TD.H2.3}
\end{align}
and
\begin{align}
  &\bigl( f_i, A_N (u_{\varepsilon,i} - u_{\varepsilon,i-1}) \bigr)_H \leq \frac{\delta_*}{4 \tau} |A_N (u_{\varepsilon,i} - u_{\varepsilon,i-1})|_H^2 + \frac{\tau}{\delta_*} |f_i|_H^2, \label{eq:TD.H2.4}
  \\
  &\qquad\qquad\qquad\qquad \mbox{ for } i = 1,2,3, \dots, n_\tau.
\end{align}

Now, taking into account \eqref{eq:TD.H2.1}--\eqref{eq:TD.H2.4}, and Lemmas \ref{Lem1} and \ref{Lem2}, it is inferred that:
\begin{subequations}\label{4.16ab}
\begin{align}
  &\frac{\delta_*}{4\tau} \sum_{i = 1}^{n_\tau} |A_N(u_{\varepsilon,i} - u_{\varepsilon,i-1} )|_H^2
%  \\
%  &\quad 
    \leq \frac{|\nabla \beta|_{[L^{\infty}(\Omega)]^N}^2}{\delta_*^2} \sum_{i=1}^{n_\tau} \frac{\delta_*}{\tau} |\nabla (u_{\varepsilon,i} - u_{\varepsilon,i-1})|_{[H]^N}^2 
    \\
    & \qquad + \frac{(N^2 + 1)|\alpha|_{L^\infty(\Omega)}^2}{\varepsilon^2 \delta_*} \cdot \tau \sum_{i=1}^{n_\tau} |u_{\varepsilon,i}|_{H^2(\Omega)}^2
%  \\
%  &\quad \quad 
    + \frac{T(N^2 + 1)|\alpha|_V^2}{\delta_*} + \frac{\tau}{\delta_*} \sum_{i=1}^{n_\tau} |f_i|_H^2
  \\
  & \leq \frac{\tilde C_2}{\varepsilon^2 \wedge 1} \bigl( |u_0|_{H^2(\Omega)}^2 + |f|_{\sH}^2 + \varepsilon^2 + 1 \bigr), \label{eq:TD.H2.5}
\end{align}
where
\begin{equation}
    \tilde C_2 := \frac{1}{\delta_*^2 \wedge 1}\Bigl( C_4|\nabla \beta|_{[L^{\infty}(\Omega)]^N}^2 + \bigl( N^2 + 1 \bigr) (T+1) \bigl( C_5 |\alpha|_{L^\infty(\Omega)}^2 + |\alpha|_V^2 \bigr) + 1 \Bigr). \label{eq:TD.H2.6}
\end{equation}
\end{subequations}\noeqref{eq:TD.H2.5, eq:TD.H2.6}
As a consequence of \eqref{embb01}, \eqref{4.16ab} and Lemma \ref{Lem1}, we arrive at:
\begin{align}
  &\frac{1}{\tau} \sum_{i=1}^{n_\tau} |u_{\varepsilon,i} - u_{\varepsilon,i-1}|_{H^2(\Omega)}^2 \leq \frac{C_0}{\tau} \sum_{i=1}^{n_\tau} \bigl( |A_N(u_{\varepsilon,i} - u_{\varepsilon,i-1})|_H^2 + |u_{\varepsilon,i} - u_{\varepsilon,i-1}|_H^2 \bigr)
  \\
  &\quad \leq \frac{1}{\varepsilon^2 \wedge 1} \cdot C_0 \left( \frac{4 \tilde C_2}{\delta_*} + C_4 \right) \bigl( |u_0|_{H^2(\Omega)}^2 + |f|_{\sH}^2 + \varepsilon^2 + 1 \bigr),
\end{align}
where 
\begin{equation}
  C_6 := C_0 \left( \frac{4 \tilde C_2}{\delta_*} + C_4 \right).
\end{equation}

Thus, the proof of this lemma is completed.
\end{proof}
\end{lem}

\noeqref{eq:TD.H2.2,eq:TD.H2.3}

\begin{proof}[Proof of Main Theorem 1]
Let $\varepsilon > 0$ be a fixed constant. Let $\tau_* \in (0,\tau_0)$ be the small constant as in Lemma \ref{Lem2}. Then, from Lemmas \ref{Lem1}--\ref{Lem3}, it is observed that
\begin{itemize}
  \item[$\sharp$3)] $\{ [u_\varepsilon]_\tau \,|\, \tau \in (0,\tau_*) \}$ is bounded in $W^{1,2}(0,T;W_0)$.
\end{itemize}
Additionally, by applying Ascoli's theorem (cf. \cite[Corollary 4]{MR0916688}), we can infer that
\begin{itemize}
    \item[$\sharp$4)] $\{ [u_\varepsilon]_\tau \,|\, \tau \in (0,\tau_*) \}$ and $\{ [\overline{u_\varepsilon}]_\tau \,|\, \tau \in (0,\tau_*) \}$ are relatively compact in $C([0,T];V)$ and weakly-$*$ compact in $L^\infty(0,T;W_0)$, respectively. 
\end{itemize}
As a consequence of $\sharp$3) and $\sharp$4), we can find a sequence $\{ \tau_n \}_{n=1}^\infty \subset (0,\tau_*)$, and a function $u_\varepsilon \in W^{1,2}(0,T;W_0)$, such that:
\begin{gather}
  \tau_* > \tau_1 > \tau_2 > \dots > \tau_n \downarrow 0 \mbox{ as } n \to \infty,
  \\
  u_{\varepsilon,n} := [u_\varepsilon]_{\tau_n} \to u_\varepsilon \mbox{ in } C([0,T];V), \mbox{ weakly in } W^{1,2}(0,T; W_0), \label{eq:conv01}
  \\
  \mbox{ and weakly-$*$ in } L^\infty(0,T;W_0), \mbox{ as } n \to \infty,
\end{gather}
and
\begin{align}
  &\overline{u}_{\varepsilon,n} := [\overline{u_\varepsilon}]_{\tau_n} \to u_\varepsilon \mbox{ in } L^\infty(0,T; V) 
  \\
  &\quad \mbox{ and weakly-$*$ in } L^\infty(0,T;W_0), \mbox{ as } n \to \infty. \label{eq:conv02}
\end{align}
    Note that this $u_\varepsilon$ satisfies Main Theorem 1 (S0)$_\varepsilon$.

    Next, let us take arbitrary open interval $I \subset (0,T)$. Then, in the light of \eqref{AP.sol}, \eqref{eq:conv01} and \eqref{eq:conv02}, one can see that:
\begin{gather}
  \int_I (\partial_t u_{\varepsilon,n}(t), \varphi)_H \,dt + \int_I (\alpha \nabla \gamma_\varepsilon(\nabla \overline{u}_{\varepsilon,n}(t)), \nabla \varphi)_{[H]^N} \,dt  
  \\
    + \int_I (\beta \nabla \partial_t u_{\varepsilon,n}(t), \nabla \varphi)_{[H]^N}\,dt =\int_I ([\overline{f}]_{\tau_n}(t), \varphi)_H \,dt, 
    \label{eq:conv03}
    \\
    \mbox{for all $ \varphi \in V $  and  $ n = 1,2,3, \dots $.}
\end{gather}
    Due to \eqref{eq:conv01}, \eqref{eq:conv02}, and (A3), letting $n \to \infty$ in \eqref{eq:conv03} yields that:
\begin{align}
  &\int_I (\partial_t u_\varepsilon(t), \varphi)_H \,dt + \int_I (\alpha \nabla \gamma_\varepsilon(\nabla u_\varepsilon(t)), \nabla \varphi)_{[H]^N} \,dt 
  \\
  &\qquad + \int_I (\beta \nabla \partial_t u_\varepsilon(t), \nabla \varphi)_{[H]^N}\,dt =\int_I (f(t), \varphi)_H \,dt, \ \mbox{for all } \varphi \in V. \label{eq:conv04}
\end{align}
Since the open interval $I \subset (0,T)$ is arbitrary, the above identity implies that $u_\varepsilon$ satisfies the equation \eqref{rx:Ev} as in Main Theorem 1 (S1)$_\varepsilon$. Moreover, the regularity $u_\varepsilon \in W^{1,2}(0,T;W_0)$ directly leads to the boundary condition \eqref{rx:BC}.

    Finally, we verify the uniqueness of the solution to (P)$_\varepsilon$. Let us take two solutions $u_\varepsilon^k \in W^{1,2}(0,T;W_0)$, $ k = 1, 2 $, corresponding to initial data $u_0^k \in W_0$ and forcing terms $f^k \in \sH, \ k = 1,2$, respectively. Here, taking the differences between the variational identities \eqref{rx:Ev} for $u_\varepsilon^k, \, k = 1,2$, and putting $\varphi = (u_\varepsilon^1 - u_\varepsilon^2)(t)$, we can see that:
\begin{align}
  &\frac{1}{2} \frac{d}{dt} \bigl| (u_\varepsilon^1 - u_\varepsilon^2)(t) \bigr|_H^2 + \frac{1}{2} \frac{d}{dt} \bigl| \sqrt{\beta} \nabla (u_\varepsilon^1 - u_\varepsilon^2)(t) \bigr|_{[H]^N}^2 \label{mark01}
  \\
  &\qquad +\int_\Omega \alpha(\nabla \gamma_\varepsilon(\nabla u_\varepsilon^1(t)) - \nabla \gamma_\varepsilon(\nabla u_\varepsilon^2(t))) \cdot \nabla (u_\varepsilon^1 - u_\varepsilon^2)(t) \,dx 
  \\
  &\qquad\leq \bigl( (f^1 - f^2)(t), (u_\varepsilon^1 - u_\varepsilon^2)(t) \bigr)_H, \ \mbox{a.e. }t \in (0,T).
\end{align}
Besides, by using the monotonicity of $\nabla \gamma_\varepsilon \subset \R^N \times \R^N$, and Young's inequality, one can deduce that:
\begin{align}
  &\frac{d}{dt} \left(\bigl| (u_\varepsilon^1 - u_\varepsilon^2)(t) \bigr|_H^2 + \bigl| \sqrt{\beta} \nabla (u_\varepsilon^1 - u_\varepsilon^2)(t) \bigr|_{[H]^N}^2\right) \label{mark02}
  \\
  &\qquad \leq |(u_\varepsilon^1 - u_\varepsilon^2)(t)|_H^2 + |(f^1 - f^2)(t)|_H^2, \ \mbox{a.e. } t \in (0,T).
\end{align}\noeqref{mark02}
Applying Gronwall's inequality, we arrive at:
\begin{gather}
  \bigl| (u_\varepsilon^1 - u_\varepsilon^2)(t) \bigr|_H^2 + \bigl| \sqrt{\beta} \nabla (u_\varepsilon^1 - u_\varepsilon^2)(t) \bigr|_{[H]^N}^2 
  \\
  \leq e^T \bigl( |u_0^1 - u_0^2|_H^2 + |\sqrt{\beta} \nabla (u_0^1 - u_0^2)|_{[H]^N}^2 + |f^1 - f^2|_\sH^2 \bigr),
    \label{mark03}
  \\
  \mbox{ for all }t \in [0,T].
\end{gather}
This implies the uniqueness of solution to (P)$_\varepsilon$.

Thus, we conclude Main Theorem 1.
\end{proof}

\begin{proof}[Proof of Main Theorem 2]
    For any $\varepsilon > 0$, let $u_\varepsilon \in W^{1,2}(0,T;W_0)$ be the unique solution to (P)$_\varepsilon$. Then, on account of \eqref{eq:conv01}, \eqref{eq:conv02}, and Lemmas \ref{Lem1} and \ref{Lem2}, we will estimate that:
  \begin{gather}
    \int_0^T \bigl( |\partial_t u_\varepsilon(t)|_H^2 + \delta_* |\nabla \partial_t u_\varepsilon(t)|_{[H]^N}^2 \bigr)\,dt  
    \\
    \leq C_4 \bigl( |u_0|_V^2 + |f|_{\sH}^2 + \varepsilon^2 + 1 \bigr), ~ \mbox{for all } \varepsilon > 0,
      \label{est:Ken01}
  \end{gather}
and
  \begin{equation}
      |u_\varepsilon|_{L^\infty(0,T;H^2(\Omega))}^2 \leq C_5 \bigl( |u_0|_{H^2(\Omega)}^2 + |f|_{\sH}^2 + \varepsilon^2 + 1 \bigr), \ \mbox{ for all } \varepsilon > 0, \label{est:Ken02}
  \end{equation}
    with the constants $C_4, C_5 > 0$ as in Lemma \ref{Lem1} and Lemma \ref{Lem2}, respectively. Also, from the assumption (A3) with \eqref{exM01} and Example \ref{Rem.ExMG}, 
  \begin{gather}
    |\nabla \gamma_\varepsilon(\nabla u_\varepsilon)| \leq 1 \mbox{ a.e. in } Q, \label{est:Ken03}
      \\[2ex]
      \alpha \nabla \gamma_\varepsilon(\nabla u_\varepsilon(t)) \in \partial \Phi_\varepsilon(\nabla u_\varepsilon(t)) \mbox{ in } [H]^N, \mbox{ for any $ \varepsilon > 0 $,  and a.e. $ t \in (0,T) $,} \label{est:Ken04}
  \end{gather} %\noeqref{est:Ken02}\noeqref{est:Ken03}
  and hence,
  \begin{align}
    &\alpha \nabla \gamma_\varepsilon(\nabla u_\varepsilon) \in \partial \widehat{\Phi}_\varepsilon^I(\nabla u_\varepsilon) \mbox{ in } [\sH]^N, \label{est:Ken0401}
    \\
    &\qquad\mbox{ for any } \varepsilon > 0, \mbox{ and any open interval }I \subset (0,T).
  \end{align}
    Since the constants $C_4$ and $C_5$ are independent of $\varepsilon$, we deduce from \eqref{est:Ken01} and \eqref{est:Ken04} that:
  \begin{itemize}
    \item[$\sharp$5)] $\{ u_\varepsilon \}_{\varepsilon \in (0,1)}$ is bounded in $W^{1,2}(0,T;V) \cap L^\infty(0,T;W_0)$, and $\{ \nabla \gamma_\varepsilon(\nabla u_\varepsilon) \}_{\varepsilon \in (0,1)}$ is included in the unit ball in $L^\infty(Q;\R^N)$.
  \end{itemize}
    So, by applying the compactness theory of Aubin's type (cf. \cite[Corollary 4]{MR0916688}), we can find a limit $u \in W^{1,2}(0,T;V) \cap L^\infty(0,T;W_0)$ of a subsequence of $\{ u_\varepsilon \}_{\varepsilon \in (0,1)}$ (not relabeled), with another limit ${\bm \omega}^* \in L^\infty(Q;\R^N)$ of $\{ \nabla \gamma_\varepsilon(\nabla u_\varepsilon) \}_{\varepsilon \in (0,1)}$, such that:
\begin{gather}
    u_\varepsilon \to u \mbox{ in } C([0,T];V), \mbox{ weakly in } W^{1,2}(0,T; V), 
    \\
    \mbox{ and weakly-$*$ in } L^\infty(0,T;W_0), \mbox{ as }\varepsilon \downarrow 0,
      \label{est:Ken05} 
  \end{gather}
    and
  \begin{equation}
    \alpha \nabla \gamma_\varepsilon(\nabla u_\varepsilon) \to \alpha {\bm \omega}^* \mbox{ weakly-$*$ in } L^\infty(Q;\R^N). \label{est:Ken07}
  \end{equation}\noeqref{est:Ken05}
%  and
%  \begin{gather}
%    |{\bm \omega}^*|_{L^\infty(Q)} \leq \varliminf_{\varepsilon \downarrow 0} |\nabla \gamma_\varepsilon(\nabla u_\varepsilon)|_{L^\infty(Q)} \leq 1. \label{est:Ken06}
%  \end{gather}
  
  In view of this, let us take a limit $\varepsilon \downarrow 0$ in \eqref{eq:conv04}. Then, by virtue of \eqref{est:Ken04}--\eqref{est:Ken07}, we can deduce that:
  \begin{align}
    &\int_I (\partial_t u(t), \varphi)_H \,dt + \int_I (\alpha {\bm \omega}^*(t), \nabla \varphi)_{[H]^N} \,dt + \int_I (\beta \nabla \partial_t u(t), \nabla \varphi)_{[H]^N}\,dt \label{est:Ken08}
    \\
    &\qquad =\int_I (f(t), \varphi)_H \,dt, \ \mbox{for any } \varphi \in V, \mbox{ and any open interval } I \subset (0,T),
  \end{align}
  and
  \begin{equation}
    u(0) = \lim_{\varepsilon \downarrow 0} u_\varepsilon(0) = u_0 \in W_0 \mbox{ in } V. \label{est:Ken09}
  \end{equation}
  Furthermore, invoking \eqref{est:Ken0401}, \eqref{est:Ken07}, Example \ref{Rem.ExMG}, and Remark \ref{Rem.MG} (\hyperlink{Fact\,1}{Fact1}), one can say that:
  \begin{equation}
    \alpha {\bm \omega}^* \in \partial \widehat{\Phi}_0^I(\nabla u) \mbox{ in } L^2(0,T;[H]^N),
  \end{equation}
  and hence,
  \begin{equation}
    \alpha {\bm \omega}^* \in \partial \Phi_0(\nabla u) \mbox{ in } [H]^N, \mbox{ for a.e. } t \in (0,T), \mbox{ and } {\bm \omega}^* \in \Sgn(\nabla u) \mbox{ a.e. in } Q. \label{est:Ken10}
  \end{equation} \noeqref{est:Ken09}
  \eqref{est:Ken08}--\eqref{est:Ken10} imply that $u$ is a solution to the problem (P)$_0$.
  
  Finally, for the verification of uniqueness, we take the solutions $u^k \in W^{1,2}(0,T;V) \cap L^\infty(0,T;W_0), \ k=1,2$, to (P)$_0$, which correspond to initial data $u_0^k \in W_0, \ k = 1,2$ and forcing terms $f^k \in \sH, \ k=1,2$, respectively. Then, with the (maximal) monotonicity of $\Sgn \subset \R^N \times \R^N$ in mind, we can apply the argument similar to \eqref{mark01}--\eqref{mark03}, and obtain the following estimate:
  \begin{align}
    &\bigl| (u^1 - u^2)(t) \bigr|_H^2 + \bigl| \sqrt{\beta} \nabla (u^1 - u^2)(t) \bigr|_{[H]^N}^2 
    \\
    &\qquad \leq e^T \bigl( |u_0^1 - u_0^2|_H^2 + |\sqrt{\beta} \nabla (u_0^1 - u_0^2)|_{[H]^N}^2 + |f^1 - f^2|_\sH^2 \bigr),
    \\
    &\hspace{5.5cm} \mbox{ for all }t \in [0,T].
  \end{align}
  This guarantees the uniqueness of solution to (P)$_0$. 
  
  Thus, the proof of Main Theorem 2 is completed.
  \end{proof}

\section{Appendix}
In this appendix, we prove some lemmas, which support the observations of $H^2$-regularities in the principal Sections 2--4.

In what follows, we use the notations as in Remark \ref{rem:NS}.

\begin{lem}\label{lem:Ap01}
  There exists a partition of unity $ \{ \tilde \eta_\ell \}_{\ell = 0}^M \subset C_c^\infty(\R^N)$ for the covering $ \{ U_\ell \}_{\ell = 0}^M $ of $\overline{\Omega}$, such that:
    \begin{equation}
      \nabla \tilde \eta_\ell \cdot n_\Gamma = 0, \ ~{\rm for}~ \ell = 0,1,2,\dots,M. \label{eq:partition.1}
    \end{equation}
\end{lem}

\begin{proof}
    First, we take a reduced open covering $ \{ U_\ell \}_{\ell = 0}^M $ of $\overline{\Omega}$, such that:
    \begin{equation}
      U_\ell' \subset \overline{U_\ell'} \subset U_\ell, \mbox{ for } \ell = 0,1,2,\dots, M, \mbox{ and } \bigcup_{\ell=1}^M U_\ell' \supset \Gamma.
    \end{equation} 
    Here, since:
    \begin{equation}
      (\Xi^{-1}_\ell \Theta_\ell^{-1}) U_\ell' \subset \overline{(\Xi^{-1}_\ell \Theta_\ell^{-1}) U_\ell'} \subset W_\ell \,\bigl(= (\Xi^{-1}_\ell \Theta_\ell^{-1}) U_\ell \bigr), \mbox{ for } \ell = 1,2,\dots, M,
    \end{equation}
    we can also take finite sets of positive constants $\{ r_\ell' \}_{\ell = 1}^M,\ \{ r_\ell'' \}_{\ell = 1}^M, \ \{ h_\ell' \}_{\ell = 1}^M$, and $\{ h_\ell'' \}_{\ell = 1}^M$, such that the following properties hold:
    \begin{equation}
      \left\{ \begin{aligned}
        r_\ell' < r_\ell'' < r_\ell,
        \\
        h_\ell' < h_\ell'' < h_\ell,
      \end{aligned} \right. ~~ \mbox{ for } \ell = 1,2,\ldots, M,
    \end{equation}
    and
    \begin{equation}
      \overline{(\Xi^{-1}_\ell \Theta_\ell^{-1}) U_\ell'} \subset W_\ell' \subset \overline{W_\ell'} \subset W_\ell'' \subset \overline{W_\ell''} \subset W_\ell, \mbox{ for } \ell = 1,2,\ldots, M,
    \end{equation}
    for the following cylindrical subdomains $W_\ell', \, W_\ell'' \subset W_\ell$:
    \begin{equation}
      \left\{ \begin{aligned}
        &W_\ell' := r_\ell' \B^{N-1} \times (-h_\ell' , h_\ell'), \\
        &W_\ell'' := r_\ell'' \B^{N-1} \times (-h_\ell'',h_\ell''),
      \end{aligned} \right. \mbox{ for } \ell = 1,\dots,M.
    \end{equation}

    For any $\ell \in \{1,\dots,M\}$, let $\zeta_\ell \in C_c^\infty(\R^N)$ be a function such that:
    \begin{equation}
      0 \leq \zeta_\ell \leq 1 ~{\rm on}~ \R^N, \ \zeta_\ell \equiv 1~{\rm on}~ W_\ell',\ \supp \zeta_\ell \subset W_\ell'', \mbox{ for } \ell =1,\ldots,M.
    \end{equation}
    Besides, we define:
    \begin{gather}
      \tilde \zeta_\ell(z', z_N) := \frac{1}{2} \bigl( \zeta_\ell(z', z_N) + \zeta_\ell(z',-z_N) \bigr),
        \\
        \mbox{ for all } z' \in \R^{N-1}, ~ z_N \in \R, \mbox{ and } \ell = 1,\dots,M.
    \end{gather}
    Clearly, 
    \begin{equation}
      0 \leq \tilde \zeta_\ell \leq 1 ~{\rm on}~ \R^N, \ \tilde \zeta_\ell \equiv 1~{\rm on}~ W_\ell',\ \supp \tilde \zeta_\ell \subset W_\ell'', \mbox{ for } \ell =1,\ldots,M.
    \end{equation}

    Next, let us define:
    \begin{equation}
      \eta_\ell(x) := \begin{dcases}
        \tilde \zeta_\ell \bigl((\Xi^{-1}_{\ell} \Theta_\ell^{-1})x \bigr), &{\rm if}~ x \in U_\ell \\
          0,& \mbox{otherwise,}
      \end{dcases} ~ \mbox{ for all } x \in \R^N, \ \ell = 1, \dots, M.
    \end{equation}
    Also, let $\eta_0 \in C_c^\infty(\R^N)$ be a function, such that:
    \begin{equation}
      0 \leq \eta_0 \leq 1 ~{\rm on}~ \R^N, \ \eta_0 \equiv 1 ~{\rm on}~ U_0',\ \supp \eta_0 \subset U_0.
    \end{equation}
    Additionally, let us take two bounded domains $\Omega', \, \Omega'' \subset \R^N$, and a function $\eta_* \in C_c^\infty(\R^N)$, such that:
    \begin{equation}
      \overline{\Omega} \subset \Omega' \subset \overline{\Omega'} \subset \Omega'' \subset \overline{\Omega''} \subset \bigcup_{\ell=0}^M U_\ell',
    \end{equation}
    and
    \begin{equation}
      0 \leq \eta_* \leq 1 ~{\rm on}~ \R^N, \ \eta_* \equiv 1 ~{\rm on}~ \overline{\Omega'},\ \supp \eta_* \subset \Omega''.
    \end{equation}
    As is easily checked,
    \begin{equation}
      \left\{ \begin{aligned}
        &\nabla \eta_\ell \cdot n_\Gamma = 0, \mbox{ for } \ell = 1,\dots,M,
        \\
        &\nabla \eta_* \cdot n_\Gamma = 0,
      \end{aligned} \right. \mbox{ on } \Gamma. \label{eq:normal01}
    \end{equation}

    Based on these, we define:
    \begin{equation}
      \tilde \eta(x) := (1 - \eta_*)(x) + \sum_{\ell=0}^M \tilde \eta_\ell(x), \mbox{ for all } x \in \R^N,
    \end{equation}
    with
    \begin{equation}
      \tilde \eta_\ell(x) := \frac{\eta_\ell(x)}{\tilde \eta(x)}, \mbox{ for all } x \in \R^N, \ \ell = 1,\dots,M. \label{eq:normal02}
    \end{equation}
    Then, having in mind:
    \begin{align}
      &\tilde \eta \in C^\infty(\R^N), \mbox{ and}
      \\
      &\tilde \eta(x) \geq \left\{\begin{aligned}
        &\eta_\ell(x) = 1, \mbox{ if } x \in U_\ell', \mbox{ for some } \ell \in \{ 0,1,\dots,M \},
        \\
        &(1 - \eta_*)(x) = 1, \mbox{ if } x \in \left( \bigcup_{\ell=0}^M U_\ell' \right)^{\rm \hspace{-0.6ex}C},
      \end{aligned} \right.
    \end{align}
    it is observed that:
    \begin{equation}
      \tilde \eta_\ell \in C_c^\infty(\R^N), \ 0 \leq \tilde \eta_\ell \leq 1 ~{\rm on}~ \R^N, \mbox{ and } \supp \tilde \eta_\ell \, (= \supp \eta_\ell)\, \subset U_\ell, \mbox{ for } \ell =1,\ldots,M.
    \end{equation}
    Also, since $1 - \eta_* \equiv 0$ on $\overline{\Omega}$, 
    \begin{equation}
      \sum_{\ell=0}^M \tilde \eta_\ell(x) = 1, \mbox{ for all } x \in \overline{\Omega}.
    \end{equation}
    Moreover, on account of \eqref{eq:normal01} and \eqref{eq:normal02}, it is verified that:
    \begin{equation}
      \nabla \tilde \eta_\ell \cdot n_\Gamma = \frac{\tilde \eta \bigl(\nabla \eta_\ell \cdot n_\Gamma\bigr) - \eta_\ell \bigl(\nabla \tilde \eta \cdot n_\Gamma\bigr)}{\tilde \eta^2} = 0, \mbox{ on } \Gamma, \mbox{ for } \ell = 0,1,\dots,M.
    \end{equation}

  Thus, we conclude this lemma.
  \end{proof}

\begin{lem}\label{lem:approx.}
  For any $v \in V$, there exists a sequence $\{ \varphi_\delta \}_{\delta > 0} \subset C^\infty(\overline{\Omega}) \cap W_0$ such that:
  \begin{equation}
    \varphi_\delta \to v \mbox{ in } V \mbox{ as }  \delta \downarrow 0, \label{eq:approx-C^infty}
  \end{equation}
  and in particular,
  \begin{equation}
    \varphi_\delta \to v \mbox{ in } H^2(\Omega) \mbox{ as }  \delta \downarrow 0, \mbox{ if } v \in W_0. \label{eq:approx-C^infty2}
  \end{equation}
\end{lem}

\begin{proof}
  Let us fix any $v \in V$. Let $\{ \tilde \eta_\ell \}_{\ell = 0}^\infty \subset C_c^\infty(\R^N)$ be the partition of unity as in Lemma \ref{lem:Ap01}. Then, to prove this lemma, it is sufficient to show the following item ($*$):
  \begin{itemize}
    \item[($*$)] for any $\ell \in \{ 0,1, \dots,M \}$, there exists a sequence of functions $\{ \varphi_\delta^{(\ell)} \}_{\delta > 0} \subset C^\infty(\overline{\Omega}) \cap W_0$, such that $\varphi_\delta^{(\ell)} \to \tilde \eta_\ell v$ in $V$ as $\delta \downarrow 0$, and in particular, when $v \in W_0$, this convergence is realized in the (strong) topology of $H^2(\Omega)$.
  \end{itemize}

  In the case when $ \ell = 0 $, the targeted function $\tilde \eta_0 v$ has a compact support in $\Omega$. So, by using the standard mollifier $\{ \rho_\delta \}_{\delta > 0}$, we will obtain the sequence $\{ \varphi_\delta^{(0)} \}_{\delta > 0}$ required in ($*$), as follows:
  \begin{equation}
    \varphi_\delta^{(0)} := \rho_\delta * (\tilde \eta_0 v) \mbox{ on } \overline{\Omega}, \mbox{ for all } \delta > 0.
  \end{equation}

    Next, we consider the case when $\ell > 0$. Let us fix  $\ell \in \{ 1,\dots, M \}$, and  set:
  \begin{equation}
    w_\ell := (\tilde \eta_\ell v) \circ \Xi_\ell^{-1} \circ \Theta_\ell^{-1} \mbox{ in } H^1(W_\ell^+).
      \label{ken.ext00}
  \end{equation}
    On this basis, we define an extension $\tilde w_\ell \in H^1(\R^N)$ of $ w_\ell $, by letting:
  \begin{gather}
    \tilde w_\ell (y', y_N) := \begin{dcases}
      w_\ell (y',y_N), & \mbox{if } y = [y',y_N] \in W_\ell^+,
      \\
      w_\ell (y',-y_N), & \mbox{if } y = [y', y_N] \in W_\ell^-,
      \\
        0, & \mbox{otherwise,}
    \end{dcases} \label{ken.ext01}
    \\
      \mbox{for a.e. } y = [y', y_N] \in \R^N, \mbox{ with } y' \in \R^{N-1} \mbox{ and } y_N \in \R.
  \end{gather}
  Since $ v \in V $ implies $\tilde \eta_\ell v \in V $, we immediately see from \eqref{ken.ext00} and \eqref{ken.ext01} that $\tilde w_\ell \in H^1(W_\ell)$. Additionally, we can say that:
    \begin{itemize}
        \item[$(**)$]if $v \in W_0$, then $\tilde w_\ell \in H^2(\R^N)$, and $ \partial_N \tilde{w}_\ell(z', 0) = 0 $ for a.e. $ z' \in \R^{N -1} $.
\end{itemize}
    In fact, by virtue of \eqref{bc0}, \eqref{eq:partition.1}, \eqref{ken.ext00}, \eqref{ken.ext01}, and $v \in W_0$, it is observed that:
  \begin{equation}
    \tilde w_\ell \in H^2(W_\ell^+ \cup W_\ell^-), \mbox{ with } \supp \tilde w_\ell \subset W_\ell, \label{ken.ext02}
  \end{equation}
    and 
    \begin{align}
        \partial_N w_\ell(z', 0) ~&  = -\bigl[ \nabla (\tilde{\eta}_\ell v) \cdot n_\Gamma \bigr](x)
        = -\bigl[ \tilde{\eta}_\ell \nabla (v \cdot n_\Gamma ) \bigr](x) -\bigl[ v \nabla (\tilde{\eta}_\ell \cdot n_\Gamma ) \bigr](x) 
        \\
        & = 0, \mbox{ for a.e. $ z' \in r_\ell \B^{N -1} $ and $ x = \Theta_\ell \Xi_\ell [z', 0] \in \Gamma $.}
        \label{ken.ext02-01}
    \end{align}
    Furthermore, for any $\varphi \in C_c^\infty(W_\ell)$, the properties \eqref{ken.ext02} and \eqref{ken.ext02-01} enable us to compute as follows:
  \begin{align}
    &\langle \partial_N^2 \tilde w_\ell, \varphi \rangle = -\int_{W_\ell} \partial_N \tilde w_\ell \partial_N \varphi dy 
    \\
    &\qquad = -\int_{r_\ell \B^{N-1}} \int_0^{h_\ell} [\partial_N w_\ell](y', y_N) \partial_N \varphi(y',y_N) \,dy_N dy'
    \\
    &\qquad\qquad + \int_{r_\ell \B^{N-1}} \int_{-{h_\ell}}^0 [\partial_N w_\ell](y', -y_N) \partial_N \varphi(y',y_N) \,dy_N dy'
    \\
    &\qquad = \int_{r_\ell \B^{N-1}} \int_0^{h_\ell} \partial_N [\partial_N w_\ell](y', y_N) \varphi(y',y_N) \, dy_N dy'
    \\
    &\qquad\qquad - \int_{r_\ell \B^{N-1}} \int_{-{h_\ell}}^0 \partial_N[\partial_N w_\ell](y', -y_N) \varphi(y',y_N) \,dy_N dy'
    \\
    &\qquad\qquad -\int_{r_\ell \B^{N-1}} \Bigl[ [\partial_N w_\ell](y', y_N) \varphi(y',y_N)\Bigr]_0^{h_\ell}\,dy' 
    \\
    &\qquad\qquad+ \int_{r_\ell \B^{N-1}} \Bigl[[\partial_N w_\ell](y', -y_N) \varphi(y',y_N)\Bigr]_{-h_\ell}^0 \,dy' 
    \\
    &\qquad = \int_{W_\ell} \biggl( [\partial_N^2 w_\ell](y',y_N) \chi_{W_\ell^+}(y',y_N) + [\partial_N^2 w_\ell](y',-y_N)\chi_{W_\ell^-}(y',y_N) \biggr)\varphi(y',y_N) \,dy.
  \end{align}
  This variational identity implies that:
  \begin{gather}
      \partial_N^2 \tilde w_\ell = \partial_N^2 w_\ell \circ \Pi_0 \in H, \mbox{ in } \D'(W_\ell), \label{ken.ext03}
    \\
    \mbox{with a Lipschitz transform } \Pi_0 :\,y = [y',y_N] \in \R^N \mapsto \Pi_0 y := [y',|y_N|].
  \end{gather}
    The item $(**)$ will be ensured as a consequence of \eqref{ken.ext02} and \eqref{ken.ext02-01}, and \eqref{ken.ext03}.

  Now, let us define a sequence $\{ \varphi_\delta^{(\ell)} \}_{\delta > 0} \in C_c^\infty(\R^N)$, as follows:
  \begin{equation}
    \varphi_\delta^{(\ell)} := (\rho_\delta * \tilde w_\ell) \circ \Xi_\ell^{-1} \circ \Theta_\ell^{-1} \mbox{ on } \overline{\Omega}, \mbox{ for all } \delta > 0. 
  \end{equation}
  Then, in view of the arguments as in \eqref{ken.ext00}--\eqref{ken.ext03}, we can say that the item ($*$) has been verified, except for the boundary condition:
  \begin{equation}
    \nabla \varphi_\delta^{(\ell)} \cdot n_\Gamma = 0 \mbox{ on } \Gamma, \mbox{ for all } \delta > 0 \mbox{ and } \ell = 1, \dots, M. \label{ken.ext04}
  \end{equation}
  included in $\{ \varphi_\delta^{(\ell)} \}_{\delta > 0} \subset C^\infty(\overline{\Omega}) \cap W_0$. However, with Remark \ref{rem:NS} in mind, this boundary condition will be verified as a consequence of the following computation:
  \begin{align}
    &\partial_N (\rho_\delta * \tilde w_\ell)(y',0) = \int_{\R^N} \partial_N \tilde w_\ell (\xi',\xi_N) \rho_\delta(y'- \xi', -\xi_N) \,d\xi
    \\
    &\qquad = \int_{r_\ell \B^{N-1}} \int_0^{h_\ell} [\partial_N w_\ell](\xi', \xi_N) \rho_\delta(y' - \xi', -\xi_N) \,d\xi_N d\xi'
    \\
    &\qquad\qquad - \int_{r_\ell \B^{N-1}} \int_{-h_\ell}^0 [\partial_N w_\ell](\xi', -\xi_N) \rho_\delta(y' - \xi', -\xi_N) \,d\xi
    \\
    &\qquad = 0, \quad \mbox{ for a.e. } y' \in r_\ell \B^{N-1}.
  \end{align}

  Thus, we complete the proof of this lemma.
\end{proof}

\begin{lem}\label{Ehrling}
  For any $r > 0$, there exists a constant $C_r \geq 0$, depending on $r$, such that:
  \begin{equation}
    | v |_{L^2(\Gamma)}^2 \leq r| \nabla v |_{[H]^N}^2 + C_r |v |^2_H, \mbox{ for all } v \in V.
  \end{equation}
\end{lem}

\begin{proof}
  We prove this lemma by a contradiction. If Lemma \ref{Ehrling} does not hold, then there exists $r_0 > 0$ such that for any $n \in \N$, there exists $v_n \in V$ satisfying:
  \begin{equation}
    |v_n|_{L^2(\Gamma)}^2 > r_0 | \nabla v_n|_{[H]^N}^2 + n |v_n |_H^2. \label{eq:Ehrling1}
  \end{equation}
  Here, let us set:
  \begin{equation}
    \hat v_n := \frac{v_n}{|v_n|_{L^2(\Gamma)}}, \quad{\rm for~any}~n \in \N.
  \end{equation}
  Then, from \eqref{eq:Ehrling1}, we can see that $\{ \hat v_n \}_{n \in \N}$ is bounded in $V$. On account of the compact embedding $V \subset L^2(\Gamma)$, there exists a subsequence $\{ n_j \}_{j \in \N} \subset \{ n \}$ and a function $\hat v \in V$ such that:
  \begin{equation}
    \hat v_{n_j} \to \hat v ~{\rm in~} L^2(\Gamma) \mbox{ as }  j \to \infty, \label{ken00}
  \end{equation}
  and hence
  \begin{equation}
    |\hat v|_{L^2(\Gamma)} = \lim_{j \to \infty} |\hat v_{n_j}|_{L^2(\Gamma)} = 1. \label{ken01}
  \end{equation}
  Meanwhile, \eqref{eq:Ehrling1} and \eqref{ken01} will lead to $\hat v = 0$ in $H$, i.e. $\hat v = 0$ in $L^2(\Gamma)$. This contradicts \eqref{ken01}. 
  
  Thus, we finish the proof of this lemma.
\end{proof}

\providecommand{\href}[2]{#2}
\providecommand{\arxiv}[1]{\href{http://arxiv.org/abs/#1}{arXiv:#1}}
\providecommand{\url}[1]{\texttt{#1}}
\providecommand{\urlprefix}{URL }

%\bibliography{sinst41}

\end{document}